\documentclass[11pt,a4paper,twoside]{amsart}
\usepackage{amsmath,amsfonts,amsthm,amsopn,color,amssymb,enumitem,cite,soul}
\usepackage{palatino}
\usepackage{graphicx}
\usepackage[normalem]{ulem}
\usepackage[colorlinks=true,urlcolor=blue,citecolor=red,linkcolor=blue,linktocpage,pdfpagelabels,bookmarksnumbered,bookmarksopen]{hyperref}
\hypersetup{urlcolor=blue, citecolor=red, linkcolor=blue}
\usepackage[left=2.5cm,right=2.5cm,top=2.5cm,bottom=2.5cm]{geometry}

\usepackage[hyperpageref]{backref}

\usepackage[colorinlistoftodos]{todonotes}
\makeatletter
\providecommand\@dotsep{5}
\def\listtodoname{List of Todos}
\def\listoftodos{\@starttoc{tdo}\listtodoname}
\makeatother

\newcommand{\e}{\varepsilon}
\newcommand{\eps}{\varepsilon}
\newcommand{\C}{\mathbb{C}}
\newcommand{\R}{\mathbb{R}}

\newcommand{\RN}{{\mathbb{R}^N}}
\newcommand{\RD}{{\mathbb{R}^2}}
\newcommand{\RT}{{\mathbb{R}^3}}
\newcommand{\RQ}{{\mathbb{R}^4}}

\newcommand{\de}{\partial}
\newcommand{\weakto}{\rightharpoonup}

\renewcommand{\le}{\leslant}
\renewcommand{\ge}{\geslant}
\renewcommand{\a }{\alpha }

\renewcommand{\b }{\beta }

\newcommand{\g }{\gamma }

\newcommand{\n }{\nabla }
\newcommand{\s }{\sigma }
\renewcommand{\t}{\theta}

\newcommand{\F}{{\mathcal F}}

\newcommand{\cd}{D^2}

\newcommand{\Hd}{H^2(\RN)}

\newcommand{\cH}{\mathcal{H}}

\newcommand{\W}{\mathcal{W}}
\newcommand{\M}{\mathcal{M}}

\newcommand{\N}{\mathbb{N}}

\renewcommand{\C}{\mathbb{C}}
\renewcommand{\o}{\omega}

\newcommand{\D }{{\mathcal D}^{1,2}(\RN)}

\newcommand{\irn }{\int_{\RN}}

\def\bbm[#1]{\mbox{\boldmath $#1$}}
\newcommand{\beq }{\begin{equation}}
\newcommand{\eeq }{\end{equation}}

\renewcommand{\le}{\leqslant}
\renewcommand{\ge}{\geqslant}

\newcommand{\cD}{\mathcal D}

\newcommand{\cU}{\mathcal{U}}

\newcommand{\fd}{\mathfrak{d}}
\newcommand{\fp}{\mathfrak{p}}
\newcommand{\fii}{\mathfrak{i}}

\def\br#1\er{\textcolor{red}{#1}} 
\def\bl#1\el{\textcolor{blue}{#1}} 

\numberwithin{equation}{section}
\newtheorem{theorem}{Theorem}[section]
\newtheorem{lemma}[theorem]{Lemma}

\newtheorem{proposition}[theorem]{Proposition}
\newtheorem{remark}[theorem]{Remark}
\newtheorem{corollary}[theorem]{Corollary}
\newtheorem*{theorem*}{Theorem} 

\newcommand{\B}{\mathbb{B}}
\newcommand{\cO}{\mathcal{O}}

\title[General mixed dispersion NLS equation]{Radial and non-radial multiple solutions to a general mixed dispersion NLS equation}

\author[P. d'Avenia]{Pietro d'Avenia}
\author[A. Pomponio]{Alessio Pomponio}
\author[J. Schino]{Jacopo Schino}

\address[P. d'Avenia and A.  Pomponio]{\newline\indent
	Dipartimento di Matematica, Meccanica e Management
	\newline\indent 
	Politecnico di Bari
	\newline\indent 
	Via Orabona 4, 70125, Bari, Italy
}
\email{\href{pietro.davenia@poliba.it}{pietro.davenia@poliba.it}}
\email{\href{alessio.pomponio@poliba.it}{alessio.pomponio@poliba.it}}

\address[J. Schino]{\newline\indent
	Department of Mathematics
	\newline\indent 
	North Carolina State University
	\newline\indent 
	2311 Stinson Drive, 27607, Raleigh, NC, USA
	\newline\indent
	and
	\newline\indent
	Institute of Mathematics
	\newline\indent
	Polish Academy of Sciences
	\newline\indent
	ul. \'Sniadeckich 8, 00-656, Warsaw, Poland
}
\email{\href{mailto:jschino@ncsu.edu}{jschino@ncsu.edu}}


\begin{document}

\begin{abstract}
We study the following nonlinear Schr\"odinger equation with a fourth-order dispersion term
\[
\Delta^2u-\beta\Delta u=g(u) \quad \text{in } \RN
\]
in the positive and zero mass regimes: in the former, $N\ge2$ and $\beta > -2\sqrt{m}$, where $m>0$ depends on $g$; in the latter, $N\ge3$ and $\beta>0$. In either regimes, we find an infinite sequence of solutions under rather generic assumptions about $g$; if $N=2$ in the positive mass case, or $N=4$ in the zero mass case, we need to strengthen such assumptions. Our approach is variational.
\end{abstract}

\maketitle


\section{Introduction}

Let $N\ge2$ and consider the following nonlinear Schr\"odinger equation with a fourth-order dispersion term
\begin{equation}\label{bilap}
\Delta^2u-\beta\Delta u=g(u) \quad \text{in } \RN,
\end{equation}
where
$g\colon\R\to\R$ satisfies 
\begin{enumerate}[label=(g\arabic{*}),ref=g\arabic{*}]
	\item \label{g1} $g$ is continuous and odd;
	
	\item \label{g2} $\displaystyle-\infty<\liminf_{s\to0}\frac{g(s)}{s}\le\limsup_{s\to0}\frac{g(s)}{s}=:-m<0$;
	
	\item\label{g3}  $\displaystyle\lim_{s\to+\infty}\frac{g(s)}{s^{2^{**}-1}}=0$  if $N\ge5$, $\displaystyle\lim_{s\to+\infty}\frac{g(s)}{e^{\alpha s^{2}}}=0$ for every $\alpha>0$ if $N=4$, with
	\[
2^{**}=
	\begin{cases}
	\displaystyle \frac{2N}{N-4} & \text{ if } N\ge5,\\
	+\infty & \text{ if } 2\le N \le 4;
	\end{cases}
	\]
	\item \label{g4} there exists $s_0\ne0$ such that $G(s_0)>0$, where $\displaystyle G(s):=\int_0^sg(t)\,dt$;
\end{enumerate}
and $$\beta>-2\sqrt{m}.$$
This kind of assumptions has been introduced in \cite{BL,BL2} for the study of the equation $-\Delta u = g(u)$. In particular, hypothesis \eqref{g2} corresponds to the so-called \emph{positive mass case}. 
\\
For $N\ge3$, we consider also the \emph{zero mass case}, i.e., when
\[
\lim_{s \to 0} \frac{g(s)}{s} = 0.
\]
Nevertheless, unlike \cite{BL,BL2}, the presence of two differential operators in \eqref{bilap} makes the assumptions about the behaviour of $g$ at the origin non-univocal, therefore we consider two sub-cases in place of \eqref{g2}:
\begin{enumerate}[label=(g\arabic{*}'),ref=g\arabic{*}']
\setcounter{enumi}{1}
\item \label{g2''} $\displaystyle -\infty < \liminf_{s\to0}\frac{g(s)}{|s|^{2^*-2}s} \le \limsup_{s\to0}\frac{g(s)}{|s|^{2^*-2}s} =: -\ell < 0$
\end{enumerate}
or
\begin{enumerate}[label=(g\arabic{*}''),ref=g\arabic{*}'']
\setcounter{enumi}{1}
\item \label{g2'} $\displaystyle \lim_{s\to0}\frac{g(s)}{|s|^{2^*-1}} = 0$,
\end{enumerate}
with $2^* = 2N/(N-2)$. 
\\
Observe that in this case $m=0$ and so we are requiring $\b >0$.
\\
We remark also that in both the positive and the zero mass case, there is no growth assumption about $g$ at infinity whenever $N < 4$.

Let us spend a few words about classical motivations to deal with \eqref{bilap}. 
\\
In the study of the nonlinear Schr\"odinger equation
\begin{equation}\label{NLS}
\mathrm{i} \de_t\psi + \Delta\psi + |\psi|^{2\s}\psi = 0, \quad \psi \colon \R \times \RN \to \C,
\end{equation}
results concerning blow-up vs. global existence and instability vs. stability appear and they depend strongly on the dimension and the nonlinearity (see, for example, \cite{Cazenave}).

%

To enlarge the range of $N$ and $\s$ for the solutions  to exist globally in time, in \cite{Karpman91,Karpman,KS} an additional fourth-order term is proposed in \eqref{NLS}, obtaining the nonlinear \emph{mixed dispersion} Schr\"odinger equation
\begin{equation}\label{MD}
\mathrm{i} \de_t\psi - \gamma \Delta^2 \psi + \Delta\psi + |\psi|^{2\s}\psi = 0, \quad \psi \colon \R \times \RN \to \C
\end{equation}
with $\gamma>0$.  
\\
Results in this direction can be found in  \cite{BouLen,FIP,Pausader}.


If \emph{standing wave solutions} to \eqref{MD}, i.e., solutions of the form $\psi(x,t) = e^{\mathrm{i}\mu t}u(x)$ for some $\mu\ge 0$ and $u \colon \RN \to \R$, are looked for, then one obtains \eqref{bilap} with $\beta = 1/\gamma$ and $g(s) = (|s|^{2\s}s - \mu s)/\g$. We recall that standing wave solutions are usually called \emph{waveguide solutions} in nonlinear optics, a major field of application of \eqref{NLS} and \eqref{MD}.

A different physical derivation for \eqref{MD} appears in \cite{FIP}. 
In nonlinear optics, it is well known that equation \eqref{NLS} can be drawn from the nonlinear Helmholtz equation
separating the fast oscillations from the slowly varying amplitude, changing the nondimensional variables, and using the {\em paraxial approximation}, which consists of neglecting some {\em small} terms.\\
The aforementioned  blow-up results for \eqref{NLS}, together with the fact that numerical simulations and asymptotic analysis of the nonlinear Helmholtz equation suggest that nonparaxiality arrests the blow-up, lead to taking into account the neglected terms. Following the standard numerical approach, a biharmonic term appears as part of the nonparaxial correction. In the end, we obtain \eqref{MD}. For more details, see \cite[Section 2]{FIP}.

This regularizing effect of the additional term $\gamma \Delta^2\psi$ appears clearly also in the Bopp--Podolsky theory (see \cite{B40,Pob42}). In this context, if $\delta_0$ is the Dirac delta function centred at the origin, the Poisson equation in $\RT$
\[
- \Delta \psi = 4 \pi \delta_0
\]
becomes
\[
 \gamma \Delta^2 \psi - \Delta \psi = 4 \pi \delta_0.
\]
The fundamental solution to the former
is $\F_1(x) := |x|^{-1}$, which is singular at $0$ and 
\[
 \int_{\RT} |\nabla \F_1|^2 \, dx = +\infty,
\]
while  the fundamental solution to the latter, instead, is $\F_2(x) := |x|^{-1}(1-e^{-|x|/\sqrt{\gamma}})$, which satisfies $\lim_{x \to 0} \F_2(x) = 1/\sqrt{\gamma}$ and
\[
\int_{\RT} \left[ \gamma(\Delta \F_2)^2  + |\nabla \F_2|^2 \right] \, dx < +\infty
\]
(see e.g. \cite[Section 3]{DS} for details).

\

Solutions to \eqref{bilap} can be found as critical points of the $C^1$ functional 
\[
I(u)  =\frac12\int_{\RN}[(\Delta u)^2+\beta|\nabla u|^2]\,dx-\int_{\RN}G(u)\,dx,
\]
with $I\colon H^2(\RN)\to\R$ in the positive mass case, $I\colon \cd(\RN)\to\R$ in the zero mass case, where $\cd(\RN)$ is the completion of $C_c^\infty(\RN)$ with respect to the norm
\[
\|u\|_{\cd} =\left(\|\Delta u\|_2^2+ \|\n u\|_2^2\right)^\frac{1}{2}.
\]

Since \eqref{bilap} is set in the whole $\RN$, we consider some symmetries in order to recover compactness. To this aim, let us recall from \cite[Definition 1.22]{Willem_Book} (see also \cite{BW,Lions82}) that a subgroup $\cO \subset \cO(N)$ is called \emph{compatible} with $\RN$ if and only if there exists $r>0$ such that
\[
\lim_{|y|\to+\infty} \mathfrak{m}(y,r) = +\infty,
\]
where $\cO(N)$ is the orthogonal group of order $N$ over $\R$ and
\[
\mathfrak{m}(y,r) := \sup\bigl\{n\ge1 : \exists \{g_i\}_{i=1}^n \subset \cO \text{ such that } i \ne j \Rightarrow B(g_iy,r) \cap B(g_jy,r) = \emptyset \bigr\}.
\]
For example, if $N = \sum_{i=1}^n N_i$ for some $n\ge1$ and $N_i\ge2$ integers, then $\cO := \prod_{i=1}^n \cO(N_i)$ is compatible with $\RN$. In particular, one can take $n=1$ and $N_1 = N$ and obtain $\cO = \cO(N)$; one can also take $n=2$ and $N_1 = N_2 = N/2$ if $N\ge4$ is even, or $n=3$, $N_1 = N_2 \le N/2 - 1$, and $N_3 = N - 2N_1$ if $N\ge6$.
\\
If $\cO$ is a subgroup of $\cO(N)$ compatible with $\RN$, we define $H^2_\cO(\RN)$ (resp. $\cd_\cO(\RN)$ when $N\ge3$) as the subspace of $\Hd$ (resp. $\cd(\RN)$) of $\cO$-invariant functions.

In order to find non-radial solutions when $N=4$ or $N\ge6$, according to the notations above we consider $n=2$ and
\[
X := \left\{ u\in\cd(\RN) : u(x_1,\dots,x_{N/2},x_{N/2+1},\dots,x_N) = -u(x_{N/2+1},\dots,x_N,x_1,\dots,x_{N/2}) \right\}
\]
if $N$ is even, or $n=3$ and
\[
X := \left\{ u\in\cd(\RN) :
\begin{array}{l}
u(x_1,\dots,x_{N_1},x_{N_1+1},\dots,x_{2N_1},x_{2N_1+1},\dots,x_N)\\
= -u(x_{N_1+1},\dots,x_{2N_1},x_1,\dots,x_{N_1},x_{2N_1+1},\dots,x_N)
\end{array} 
\right\}
\]
if $N\ge6$, and define $H_X^2(\RN) := H_\cO^2(\RN) \cap X$ (resp. $\cd_X(\RN) := \cd_\cO(\RN) \cap X$), where $\cO = \cO(N/2) \times \cO(N/2)$ in the former case and $\cO = \cO(N_1) \times \cO(N_1) \times \cO(N_3)$ in the latter. It is clear that $X \cap \cd_{\cO(N)}(\RN) = \{0\}$.

For simplicity, when there is no risk of misunderstanding, we introduce the notations
\begin{equation*}
\cH := H_{\cO(N)}^2(\RN) \quad \text{or} \quad \cH := H_X^2(\RN)
\end{equation*}
and
\begin{equation*}
\cD := \cd_{\cO(N)}(\RN) \quad \text{or} \quad \cD := \cd_X(\RN),
\end{equation*}
the right-hand ones provided $N=4$ or $N\ge6$. This means that, whenever a statement is made for $\cH$ (resp. $\cD$), it holds both for $H_{\cO(N)}^2(\RN)$ (resp. $\cd_{\cO(N)}(\RN)$) and, if $N=4$ or $N\ge6$, for $H_X^2(\RN)$ (resp. $\cd_X(\RN)$).

As it is well known, we can work in such subspaces of $H^2(\RN)$ or $\cd(\RN)$ and still find solutions to \eqref{bilap} in virtue of the principle of symmetric criticality \cite{Palais}.

Now we state our results, beginning with the positive mass regime.
\begin{theorem}\label{th:main}
Assume that $N\ge3$ and \eqref{g1}--\eqref{g4} hold. Then there exists a sequence $\{u_n\} \subset \cH$ of solutions to \eqref{bilap} such that $I(u_n) \to +\infty$ as $n \to +\infty$.
\end{theorem}

In the two dimensional case we have to require stronger assumptions about $g$. More precisely we have what follows.
\begin{proposition}\label{pr:main}
Let $N=2$. Assume that  \eqref{g1}, \eqref{g2}, and \eqref{g4} hold and that
\begin{equation}\label{AR}
\text{there exists $\gamma>2$ such that $\displaystyle g(s)s + ms^2 \ge \gamma \left(G(s) + \frac{m}{2}s^2\right)$ for every $s\in\R$.}
\end{equation}
Then there exists a sequence $\{u_n\} \subset H_{\cO(2)}^2(\RD)$ of solutions to \eqref{bilap} such that $I(u_n) \to +\infty$ as $n \to +\infty$.
\end{proposition}
We point out that \eqref{AR} holds, for example, if $g(s) = \a |s|^{2\s}s - m s$ for some $\s,\a > 0$.

In the zero mass regime (when $N\ge3$), taking into account the two different subcases mentioned above, we have the following results.

%
%

\begin{theorem}\label{th:mm0}
Let $N\ge 3$. Assume that \eqref{g1}, \eqref{g2''} or \eqref{g2'}, \eqref{g3}, and \eqref{g4} hold. Then there exists a sequence $\{u_n\} \subset \cd_{\cO(N)}(\RN)$ of solutions to \eqref{bilap} such that $I(u_n) \to +\infty$ as $n \to +\infty$.\\
If, moreover, $N\ge6$, then there exists a sequence $\{u_n\} \subset \cd_X(\RN)$ of solutions to \eqref{bilap} such that $I(u_n) \to +\infty$ as $n \to +\infty$.
\end{theorem}

Observe that these two last theorems do not deal with the non-radial setting whenever $N~\!=~\!4$.  
The reason is that assumptions \eqref{g1}, \eqref{g2''} or \eqref{g2'}, and \eqref{g3} seem to be sufficient to prove that the energy functional $I$ is well defined (and of class $C^1$) only over $\cd_{\cO(4)}(\R^4)$, as a consequence of a new Adams-type inequality proved  in Lemma \ref{le:radiallemma} and Corollary \ref{corTh14RS-zero}. 
However, we do not know if this holds in  the whole space $\cd(\R^4)$.  
Thus,
we strengthen (\ref{g3}) and the following holds.



\begin{theorem}\label{th:mainnr}
Let $N=4$ and assume that \eqref{g1}, \eqref{g2''} or \eqref{g2'}, \eqref{g4}, and
\begin{enumerate}[label=(g\arabic{*}'),ref=g\arabic{*}']
\setcounter{enumi}{2}
\item \label{g3'} 
$\displaystyle\lim_{s\to+\infty}\frac{g(s)}{e^{\a s^{4/3}}} = 0$, for every $\a >0$
\end{enumerate}
hold. Then there exists a sequence $\{u_n\} \subset \cd_X(\RQ)$ of solutions to \eqref{bilap} such that $I(u_n) \to +\infty$ as $n \to +\infty$.
\end{theorem}

Recently many authors focused their attention on the nonlinear Schr\"odinger equation with a fourth-order dispersion term in all of $\RN$. Here we recall just some of them. Existence and properties of ground states, multiplicity of solutions, normalized solutions, and (in)stability have been considered in \cite{BCdN,BCGJ1,BCGJ2,BN,BFJ,FJMM,FLNZ,LW}, while \cite{BCM,LCW,MSV,SCW,YT,YW,ZTZ} studied the mixed  dispersion nonlinear Schr\"odinger equation in the non-autonomous case and with different types of nonlinearities.

Nevertheless, up to our knowledge, this is the first work  where this problem is tackled in presence of very general nonlinearities and, in particular, it seems that the zero mass case has not been considered so far.  Furthermore, since our nonlinearity satisfies very general assumptions, we cannot adapt easily the strategies of the aforementioned papers. For example, in \cite{BN}, to find a least-energy solution, the authors minimize the energy functional over the set
\[
\left\{u \in \Hd : \irn |u|^{2\s+2} \, dx = 1\right\},
\]
scaling the obtained minimizer $u \mapsto \t u$ for a suitable $\t>0$. However, the inhomogeneity of our nonlinearity makes it impossible to use such an  approach. In addition,  the presence of two differential terms of different orders (unless $\beta=0$ in the positive mass regime) prevents us also from using internal scaling $u \mapsto u(\t\cdot)$.

Furthermore, it is hard to prove the boundedness of Palais--Smale sequences.  In order to overcome such a difficulty, inspired by \cite{jj}, we introduce a two-variable functional: this allows to construct a suitable Palais--Smale sequence which, in addition, {\em almost} satisfies a Poho\v zaev-type identity. 
When dealing with such a particular bounded sequence, we also need to overcome the lack of compactness. 
In the radial setting, this is usually done using the well known Radial Strauss Lemma \cite{BL,Strauss}. However, since we are also interested in non-radial solutions, we develop a unified approach, inspired by \cite{M-2020,M-2021}, which holds in both cases and is only based on the symmetry structure introduced before.

Moreover, since we are interested in multiplicity results, we also have  to find a sequence of mini--max levels that diverges positively. To this aim, we follow two  strategies according to the different assumptions. 
More precisely,  for the positive mass case and for the zero mass case when \eqref{g2''} holds, we adapt an argument of \cite{HIT} introducing a comparison functional. Under the assumption \eqref{g2'}, instead,  we proceed in a way similar to \cite{CGT,HT,IT}, proving a suitable deformation lemma exploiting, once again, the two-variable functional.

The paper is organized as follows: we deal with the positive mass case in Section 2 and the zero mass case in Section \ref{3}; in particular, in both sections we start with the functional framework, then we show some compactness results, and finally we prove the main theorems. We conclude with some open questions in Section \ref{4}.

\

{\bf Notations}
\\
For $1 \le p \le +\infty$, we denote the usual $L^p(\RN)$ norm by $\|\cdot\|_p$.\\
For $y \in \RN$ and $r>0$, we denote $B(y,r) := \{x \in \RN : |x-y| < r\}$ and $B_r := B(0,r)$.\\
For every integer $k\ge1$, $\B^k\subset\R^k$ is the closed unit ball centred at the origin, while $\mathbb{S}^{k-1} := \partial\B^k$.\\
If $\o_{N-1}$ denotes the $(N-1)$-dimensional measure of $\mathbb{S}^{N-1}$, then we recall that $\o_3 = 2\pi^2$.\\
The letters $c$ and $C$  denote positive constants that may change after an inequality sign and whose precise value is not relevant.

\section{The positive mass case}\label{2}

\subsection{The functional framework}
\

As observed in \cite{BN}, if $\beta>-2\sqrt{m}$, then fixing $m'\in(0,m)$ such that $\beta>-2\sqrt{m'}$,
\[
\|u\|:=\left(
\| \Delta u\|_2 ^2+ \beta \| \nabla u\|_2 ^2+ m' \| u \|_2^2\right)^\frac 12
\]
defines a norm in $\Hd$, which is equivalent to the standard one. 
Concerning $\cd(\RN)$, by \cite{DS} we know that
\[
\cd(\RN) = \big\{u\in \D \mid \Delta u\in L^2(\RN)\big\}.
\]
In particular, $\cd(\RN)$ is continuously embedded in $\D $.

\begin{proposition}\label{premb}
For any $N\ge 3$,  $\cd(\RN)$  is continuously embedded into $ \W^{1,2^*}(\RN)$. 
\end{proposition}

\begin{proof}
Let  $u\in C_c^\infty(\RN)$. By Sobolev inequality, there exists $C>0$ such that
\[
\|u\|_{2^*}\le C\|\n u\|_2;
\qquad
\|\de_i u\|_{2^*}\le  C\|\n \de_i u\|_2
\text{ for any }i=1,\ldots,N.
\]
Moreover, being 
\[
\sum_{i,j} \irn |\de_{ij}u|^2 dx
=\irn |\Delta u|^2 dx,
\]
we deduce that
$$\|\n u\|_{2^*}\le C\|\Delta u\|_2.$$
Therefore, for any $u\in C_c^\infty(\RN)$ we have
\begin{equation}\label{emb5}
\|u\|_{\W^{1,2^*}} = \left(\|\n u\|_{2^*}^{2^*} + \|u\|_{2^*}^{2^*}\right)^{1/2^*}\le C\big(\|\Delta u\|_2 +\|\n u\|_2\big).
\end{equation}
Now let $u\in \cd(\RN)$ and $\{u_n\}$ be a sequence in $C_c^\infty(\RN)$ such that $u_n \to u$ in $\cd(\RN)$. Then, using also the continuous embedding of $\cd(\RN)$ into $\D $ we get that, up to a subsequence, $u_n \to u$ and $|\n u_n|\to |\n u|$ a.e. in $\RN$. 
Moreover, by \eqref{emb5} and using Fatou's Lemma, we deduce that 
	\[
	\|u\|_{\W^{1,2^*}}\le C\|u\|_{\cd},
	\]
	and so $\cd(\RN) \hookrightarrow \W^{1,2^*}(\RN)$.
\end{proof}

As an immediate consequence we have the following
\begin{corollary}\label{coemb}
The following continuous embeddings hold.
\begin{enumerate}
\item If $N\ge 5$, then $\cd(\RN) \hookrightarrow L^s(\RN)$, for any $s\in [2^*,2^{**}]$.
\item $\cd(\R^4) \hookrightarrow L^s(\R^4)$, for any $s\in [4,+\infty) $.
\item $\cd(\RT) \hookrightarrow L^s(\RT)$, for any $s\in [6,+\infty] $.
\end{enumerate}
\end{corollary}

\begin{proof}
We already know that, by Proposition \ref{premb}, $\cd(\RN)\hookrightarrow  \W^{1,2^*}(\RN)$.
\\
If $N\ge 5$, since $2^{**}=(2^*)^*$, and so $ \W^{1,2^*}(\RN)\hookrightarrow L^{2^{**}}(\RN)$, we can conclude.
\\
The cases $N=3$ and $N=4$ follow immediately.
\end{proof}

We remark that the case $N=3$ has been already proved in \cite[Lemma 3.1]{DS}.

When $N=4$, let us recall the following sharp result (i.e., \cite[Theorem 1.4]{RS}), which we write explicitly for $H^2(\RQ)$.
\begin{lemma}\label{Th14RS}
There exists  $C>0$ such that
\[
\sup_{u\in H^2(\RQ),\ \| u\|\le 1} \int_{\RQ} \left(e^{32\pi^2 u^2}-1\right)\,dx\le C.
\]
\end{lemma}

As a consequence of Lemma \ref{Th14RS} we have
\begin{corollary}\label{corTh14RS}
Let $\sigma\ge 2$, $M>0$, and $\alpha>0$ such that $\alpha M^2<32\pi^2$. Then there exists $C>0$ such that for every $\tau\in \left(1,32\pi^2/(\alpha M^2)\right]$ and $u\in H^2(\RQ)$ with $\|u\|\le M$,
\[
\int_{\RQ} |u|^\sigma \left(e^{\alpha u^2}-1\right)\,dx\le C\|u\|_\frac{\sigma\tau}{\tau-1}^\sigma.
\]
\end{corollary}
\begin{proof}
First, observe that, if $s\ge0$ and $t\ge1$,
\begin{equation}\label{ineqe1}
(e^s-1)^t\le e^{st}-1.
\end{equation}
Let $u\in H^2(\RQ)$. By H\"older inequality and \eqref{ineqe1} we have that, for every $\tau>1$,
\[
\int_{\RQ} |u|^\sigma \left(e^{\alpha u^2}-1\right)\,dx
\le
\|u\|_{\frac{\sigma\tau}{\tau-1}}^\sigma
\left(\int_{\RQ} \left(e^{\alpha u^2}-1\right)^\tau\,dx\right)^{1/\tau}
\le
\|u\|_{\frac{\sigma\tau}{\tau-1}}^\sigma
\left(\int_{\RQ} \left(e^{\alpha\tau u^2}-1\right)\,dx\right)^{1/\tau}.
\]
Moreover, if $\tau\in \left(1,32\pi^2/(\alpha M^2)\right]$ and $\|u\|\le M$, by Lemma \ref{Th14RS},
\begin{align*}
\int_{\RQ} \left(e^{\alpha\tau u^2}-1\right)\,dx
&=
\int_{\RQ} \left(e^{\alpha\tau\|u\|^2 (u/\|u\|)^2}-1\right)\,dx
\le
\int_{\RQ} \left(e^{\alpha\tau M^2 (u/\|u\|)^2}-1\right)\,dx\\
&\le
\int_{\RQ} \left(e^{32\pi^2 (u/\|u\|)^2}-1\right)\,dx
\le C
\end{align*}
and we conclude.
\end{proof}

\begin{remark}
Corollary \ref{corTh14RS} remains valid for $0<\sigma<2$ provided $\frac{\s\tau}{\tau-1}\ge2$.
\end{remark}

\subsection{Some compactness results}
\

In this section, we prove some useful compactness results that we will apply later.

We begin with the following variant of Lions's Lemma \cite[Lemma I.1]{Lions84_2}.

\begin{lemma}\label{le:Lions}
Let $N\ge2$, and $F\colon\R\to\R$ be a continuous function such that
\begin{equation}
\label{Fin0}
\lim_{s \to 0} \frac{F(s)}{s^2} = 0
\end{equation}
and
\begin{alignat}{2}
&\lim_{|s|\to+\infty}\frac{F(s)}{|s|^{2^{**}}} = 0&&\text{ if } N\ge5,\label{Finfty5}\\
&\lim_{|s|\to+\infty}\frac{F(s)}{e^{\alpha s^2}} = 0
\ \text{ for all } \alpha>0\quad
&& \text{ if }N=4.\label{Finfty4}
\end{alignat}
Assume that $\{u_n\}\subset\Hd$ is bounded and there exists $r>0$ such that
\begin{equation*}
\lim_n\sup_{y\in\RN}\int_{B(y,r)}u_n^2 \, dx = 0.
\end{equation*}
Then
\[
\lim_n\irn |F(u_n)| \, dx = 0.
\]
\end{lemma}
\begin{proof}
First, let us consider the case $N\ge 5$.\\
By \eqref{Fin0} and \eqref{Finfty5}, for every $p\in(2,2^{**})$ and $\e>0$ there exists $c_\e>0$ such that, for all $s\in\R$,
	\begin{equation*}
		|F(s)| \le \e(s^2 + |s|^{2^{**}}) + c_\e|s|^p.
	\end{equation*}
	Since $\{u_n\}$ is bounded in $L^2(\RN)$, and, by Sobolev embeddings, 	it is also bounded in $L^{2^{**}}(\RN)$, there exists $C>0$ such that, for every $n\in\N$,
	\[
	\irn |F(u_n)| \, dx \le C\e + c_\e\|u_n\|_p^p.
	\]
	Thus it suffices to prove that $u_n\to0$ in $L^p(\RN)$ at least for one $p\in (2,2^{**})$.\\
	Let us take $p=2(1+4/N)$.\\
	From the interpolation inequality for Lebesgue spaces and Sobolev inequality we have that, for every $y\in\RN$ and $r>0$ as in the statement,
	\[
	\|u_n\|_{L^p(B(y,r))} \le \|u_n\|_{L^2(B(y,r))}^{1-\lambda}\|u_n\|_{L^{2^{**}}(B(y,r))}^\lambda \le C\|u_n\|_{L^2(B(y,r))}^{1-\lambda}\|u_n\|_{H^2(B(y,r))}^\lambda,
	\]
	where $C>0$ does not depend on $y\in\RN$ and $\lambda=2/p=N/(N+4)$. Hence
	\[
	\|u_n\|_{L^p(B(y,r))}^p
	\le C\|u_n\|_{L^2(B(y,r))}^{p-2}\|u_n\|_{H^2(B(y,r))}^{2}.
	\]
	Then, covering $\RN$ with balls of radius $r$ such that each point is contained in at most $N+1$ balls we obtain
	\[
	\|u_n\|_p^p
	\le C \sup_k\|u_k\|^2 \sup_{y\in\RN}\left(\int_{B(y,r)}|u_n|^{2} \, dx\right)^{(p-2)/2}\to0.
	\]
	If $N=4$, by \eqref{Fin0} and \eqref{Finfty4},  for every $\e>0$, $\alpha>0$, and $\sigma \ge 2$ there exists $c_\e>0$ such that, for all $s\in\R$,	
	\begin{equation*}
		|F(s)| \le \e s^2 + c_\e|s|^\sigma(e^{\alpha s^2}-1).
	\end{equation*}
	Then, applying Corollary \ref{corTh14RS}, the boundedness of $\{u_n\}$ implies that for $\alpha>0$ and $\tau>1$ such that $\alpha\tau\sup_{n}\|u_n\|^2 \le 32\pi^2$
	\[
	\int_{\RQ} |F(u_n)| \, dx
	\le
	C \varepsilon+c_\e \|u_n\|_\frac{\sigma\tau}{\tau-1}^\sigma
	\]
	and so it suffices to prove that $u_n\to0$ in $L^\frac{\sigma\tau}{\tau-1}(\RN)$ at least for one couple $(\sigma,\tau)$ with $\sigma \ge 2$ and $\tau>1$. For example, to simplify the computations, we take $\s = 3$ and $\tau = 5$.\\
	Arguing as before, by interpolation we have that for every $y\in\RN$,
	\[
	\|u_n\|_{L^\frac{15}{4}(B(y,r))}
	\le
	\|u_n\|_{L^2(B(y,r))}^{1-\lambda}
	\|u_n\|_{L^{16}(B(y,r))}^\lambda
	\le C\|u_n\|_{L^2(B(y,r))}^{1-\lambda}\|u_n\|_{H^2(B(y,r))}^\lambda,
	\]
	where $C>0$ does not depend on $y\in\RN$ and $\lambda=8/15$, which allows us to conclude that $\|u_n\|_\frac{15}{4}\to 0$.\\
	Finally, if $N\in\{2,3\}$, by \eqref{Fin0}, using the boundedness of $\{u_n\}$, we can write
	\begin{equation*}
		|F(s)| \le \e s^2 + c_\e|s|^3 \quad \text{for all } s\in\left[-\sup_n\|u_n\|_\infty,\sup_n\|u_n\|_\infty\right],
	\end{equation*}
	and so
	\[
	\irn |F(u_n)| \, dx \le C\e + c_\e\|u_n\|_3^3.
	\]
	To prove that $u_n \to 0$ in $L^3(\mathbb{R}^N)$ we apply again the interpolation inequality, obtaining that, for every $y\in\RN$,
	\[
	\|u_n\|_{L^3(B(y,r))} \le \|u_n\|_{L^2(B(y,r))}^{1-\lambda}
	\|u_n\|_{L^{4}(B(y,r))}^\lambda \le c\|u_n\|_{L^2(B(y,r))}^{1-\lambda}\|u_n\|_{H^2(B(y,r))}^\lambda,
	\]
	where $c>0$ does not depend on $y\in\RN$ and $\lambda=2/3$, and we conclude as before.
\end{proof}

\begin{remark}
The condition $\displaystyle\lim_n \sup_{y\in\RN} \int_{B(y,r)} u_n^2 \, dx = 0$ holds if $\displaystyle\lim_n \sup_{y\in\RN} \int_{B(y,r)} |u_n|^q \, dx = 0$ for some $q \in [2,2^*)$.
\end{remark}

The next lemma shows when the condition $\lim_n\sup_{y\in\RN}\int_{B(y,r)}u_n^2 \, dx = 0$ can occur.

\begin{lemma}\label{le:concen}
Let $\cO\subset\cO(N)$ be a subgroup compatible with $\RN$, with $r>0$ as in the definition of compatibility. Let $(Y,\|\cdot\|_Y)$ be a normed space such that $Y\hookrightarrow L^2_\textup{loc}(\RN)$ compactly and $Y\hookrightarrow L^q(\RN)$ for some $q\in[2,+\infty)$. If $\{u_n\}\subset Y$ is bounded, $u_n \to 0$ a.e. in $\RN$, and each $u_n$ is $\cO$-invariant, then
\[
\lim_n\sup_{y\in\RN}\int_{B(y,r)}u_n^2 \, dx = 0.
\]
\end{lemma}
\begin{proof}
Since each $u_n$ is $\cO$-invariant, for every $n$ we have
\[
\mathfrak{m}(y,r) \left(\int_{B(y,r)} u_n^2 \, dx\right)^\frac{q}{2}
\le C \mathfrak{m}(y,r) \int_{B(y,r)} |u_n|^q \, dx 
\le C \|u_n\|_q^q 
\le C\|u_n\|_Y^q 
\le C,
\]
where $C>0$ does not depend on $y$.\\
Let $\e>0$.
Since $\cO$ is compatible with $\RN$, there exists $R>0$ such that for every $n$
\[
\sup_{|y|>R} \int_{B(y,r)} u_n^2 \, dx \le \e.
\]
Moreover, from the compact embedding $Y \hookrightarrow L^2_\textup{loc}(\RN)$ and the almost everywhere pointwise convergence $u_n\to0$, for every sufficiently large $n$
\[
\sup_{|y|\le R} \int_{B(y,r)} u_n^2 \, dx \le \int_{B(0,R+r)} u_n^2 \, dx \le \e.\qedhere
\]
\end{proof}

\begin{remark}
Observe that, for instance, we will apply Lemma \ref{le:concen} when $Y = \cd(\RN)$ and $q=2^*$ (if $N\ge3$), or $Y = H^2(\RN)$ and $q=2$.
\end{remark}

Now we prove the following compactness result (see \cite{M-2020, M-2021}).

\begin{proposition}\label{pr:Jarek}
Let $F\in C^1(\RN)$ be such that $F(0)=0$ and
\begin{itemize}
	\item if $N\ge5$, there exists $C>0$ such that
	\begin{equation*}
		|F'(s)| \le C\left(|s| + |s|^{2^{**}-1}\right) \quad \text{for all } s\in\R;
	\end{equation*}
	\item if $N=4$, for every $\alpha>0$ there exist $\sigma \ge2$ and $C>0$ such that
	\begin{equation*}
	|F'(s)| \le C\left(|s| + \big(e^{\alpha s^2}-1\big)|s|^{\sigma-1}\right) \quad \text{for all } s\in\R;
	\end{equation*}
	\item if $N\in\{2,3\}$, there exists $C>0$ such that
	\begin{equation*}
		|F'(s)| \le C|s| \quad \text{for all } s\in[-1,1].
	\end{equation*}
\end{itemize}
Let $\{u_n\}\subset \Hd$ be bounded and such that $u_n \to u_0$ a.e. in $\RN$ for some $u_0\in\Hd$.
Then
\begin{equation}\label{conv1}
\lim_n \irn \big(F(u_n) - F(u_n-u_0) \big)\, dx = \irn F(u_0) \, dx.
\end{equation}
If, in addition,
\begin{alignat*}{3}
&\lim_{s \to 0} \frac{F(s)}{s^2} = \lim_{|s|\to+\infty} \frac{F(s)}{|s|^{2^{**}}} = 0 && \text{ when } N\ge5,\\
&\lim_{s \to 0} \frac{F(s)}{s^2} = \lim_{|s|\to+\infty} \frac{F(s)}{e^{\alpha s^{2}}} = 0  \ \ \text{ for all } \alpha>0 \ \ \ &&\text{ when } N=4,\\
&\lim_{s \to 0} \frac{F(s)}{s^2} = 0 && \text{ when } N\in\{2,3\},
\end{alignat*}
and $u_0$ and all the $u_n$ are $\cO$-invariant for a suitable subgroup $\cO\subset\cO(N)$ compatible with $\RN$, then
\begin{equation*}
\lim_n \irn F(u_n) \, dx = \irn F(u_0) \, dx.
\end{equation*}
\end{proposition}
\begin{proof}
Let us begin with the case $N\ge5$.\\
Note preliminarily that for every measurable $\Omega\subset\R^N$ and every $t\in[0,1]$
\[\begin{split}
\int_\Omega \big|F'\bigl(u_n + (t-1)u_0\bigr)u_0\big| \, dx & \le C\int_\Omega \bigl(|u_n + (t-1)u_0| + |u_n + (t-1)u_0|^{2^{**}-1}\bigr)|u_0| \, dx\\
& \le C\left( \||u_n| + |u_0|\|_2 \|u_0\|_{L^2(\Omega)} + \||u_n| + |u_0|\|_{2^{**}}^{2^{**}-1} \|u_0\|_{L^{2^{**}}(\Omega)} \right)\\
& \le C\left(\|u_0\|_{L^2(\Omega)} + \|u_0\|_{L^{2^{**}}(\Omega)}\right)
\end{split}\]
for some $C>0$ that does not depend on $n$ or $\Omega$. Therefore using Vitali's Theorem we obtain
\[\begin{split}
\irn \big(F(u_n) - F(u_n-u_0)\big) \, dx & = \irn \int_0^1 F'\bigl(u_n + (t-1)u_0\bigr)u_0 \, dt \, dx\\
& \to \irn \int_0^1 F'(tu_0)u_0 \, dt \, dx = \irn F(u_0) \, dx,
\end{split}\]
and so \eqref{conv1} is proved.\\
If $N=4$, for every measurable $\Omega\subset\R^N$, every $t\in[0,1]$, $\alpha>0$, and for  $\sigma\ge 2$  as in the assumptions, there holds
\[
\int_\Omega \big|F'\bigl(u_n + (t-1)u_0\bigr)u_0\big| \, dx \le C\int_\Omega \left(|u_n| + |u_0| + \left(e^{\alpha(|u_n| + |u_0|)^2} - 1\right) (|u_n| + |u_0|)^{\sigma-1}\right)|u_0| \, dx.
\]
Obviously
\[
\int_\Omega (|u_n| + |u_0|)|u_0| \, dx \le C\|u_0\|_{L^2(\Omega)}
\]
for some $C>0$ that does not depend on $n$ or $\Omega$. Moreover, let us write $v_n := |u_n| + |u_0|$ and let $M>0$ be such that $\|v_n\| \le M$.
We can choose $\a>0$ and  $p_1,p_2,p_3>1$ such that   $1/p_1 + 1/p_2 + 1/p_3 = 1$, $\a p_1 M^2\le 32\pi^2$,  $p_2 \ge 2/(\sigma-1)$, and $p_3\ge2$,  so that, from Lemma \ref{Th14RS}, the sequence $\{e^{\alpha p_1 v_n^2} - 1\}$ is bounded in $L^1(\RQ)$, obtaining
\[\begin{split}
\int_\Omega \left(e^{\alpha v_n^2} - 1\right) v_n^{\sigma-1}|u_0| \, dx & \le \left(\int_{\RQ} \left(e^{\alpha v_n^2} - 1\right)^{p_1}  dx\right)^{1/p_1} \|v_n\|_{(\sigma-1)p_2}^{\sigma-1} \|u_0\|_{L^{p_3}(\Omega)}\\
& \le \left(\int_{\RQ}\big( e^{\alpha p_1 v_n^2} - 1\big) dx\right)^{1/p_1} \|v_n\|_{(\sigma-1)p_2}^{\sigma-1} \|u_0\|_{L^{p_3}(\Omega)}
\\
&\le C'\|u_0\|_{L^{p_3}(\Omega)}
\end{split}\]
for some $C'>0$ not depending on $n$ and concluding as before. Note that such a choice of $\a, p_1,p_2,p_3$ is possible by taking $\a $ sufficiently small, $p_1$ sufficiently close to $1$, and $p_2,p_3$ sufficiently large.  
\\
Finally, if $N\in\{2,3\}$, in view of the embedding $\Hd\hookrightarrow L^\infty(\RN)$, there exists $T>0$ such that $\sup_n\|u_n\|_\infty \le T$ and $\widetilde C = \widetilde C(T)>0$ such that
\[
|F'(s)| \le \widetilde C|s| \quad \text{for all } s\in[-2T,2T].
\]
Hence, in a similar way as above, for every measurable $\Omega\subset\R^N$ and every $t\in[0,1]$, we get
\[
\int_\Omega \big|F'\bigl(u_n + (t-1)u_0\bigr)u_0\big| \, dx \le C \|u_0\|_{L^2(\Omega)}
\]
for some $C>0$ that does not depend on $n$ or $\Omega$ and conclude as before.
\\
Now let us move to the second part, assuming that all the $u_n$ and $u_0$ are $\cO$-invariant. Since \eqref{conv1} holds, it is enough to prove that
\[
\irn F(u_n - u_0) \, dx \to 0,
\]
but this is true in virtue of Lemmas \ref{le:Lions} and \ref{le:concen}.
\end{proof}

Applying Proposition \ref{pr:Jarek} to the function $F(s) = |s|^p$, we get\begin{corollary}\label{co:emb}
Let $N\ge2$ and $\cO\subset\cO(N)$ a subgroup compatible with $\RN$. Then $H^2_\cO(\RN) \hookrightarrow\hookrightarrow L^p(\RN)$ for every $p\in(2,2^{**})$.
\end{corollary}

In a similar way to Proposition \ref{pr:Jarek}, the following further compactness result  for $F'$ holds.

\begin{proposition}\label{pr:Jarek'}
Let $F\in C^1(\RN)$ be such that $F(0)=0$ and
\begin{alignat*}{2}
	& \lim_{s \to 0} \frac{F'(s)}{|s|} = \lim_{|s|\to+\infty} \frac{F'(s)}{|s|^{2^{**}-1}} = 0 && \text{ when } N\ge5,\\
	& \lim_{s \to 0} \frac{F'(s)}{|s|} = \lim_{|s|\to+\infty} \frac{F'(s)}{e^{\alpha s^{2}}} = 0 \ \ \text{ for all } \alpha>0 \ \ \ &&  \text{ when } N=4,\\
	& \lim_{s \to 0} \frac{F'(s)}{|s|} = 0 && \text{ when } N\in\{2,3\},
\end{alignat*}
and let $\{u_n\}$ be a bounded sequence of $\cO$-invariant functions in  $\Hd$, for a suitable subgroup  $\cO\subset\cO(N)$ compatible with $\RN$, such that $u_n \to u_0$ a.e. in $\RN$ for some $u_0\in\Hd$.\\
Then
\begin{equation*}
	\lim_n \irn F'(u_n)u_n \, dx = \irn F'(u_0)u_0 \, dx.
\end{equation*}
\end{proposition}
\begin{proof}
	As in the proof of Proposition \ref{pr:Jarek}, from Vitali's Theorem
	\[\begin{split}
		\left|\irn \big(F'(u_n)u_n - F'(u_0)u_0\big)  dx\right| & \le \irn |F'(u_n)-F'(u_0)||u_0| \, dx + \irn |F'(u_n)||u_n-u_0| \, dx\\
		& = o_n(1) + \irn |F'(u_n)||u_n-u_0| \, dx.
	\end{split}\]
	Fix $p\in(2,2^{**})$.
	Since $u_0$ and all the $u_n$ are $\cO$-invariant, from Corollary \ref{co:emb} we deduce that $\lim_n \|u_n - u_0\|_p = 0$.\\
	Assume first that $N\ge 5$ and let $\e>0$. There exists $c_\e>0$ such that for every $s\in\R$
	\[
	|F'(s)| \le \e(|s| + |s|^{2^{**}-1}) + c_\e|s|^{p-1}.
	\]
	Whence there exists $C>0$ not depending on $\e$ such that for every sufficiently large $n$
	\[
		\irn |F'(u_n)||u_n-u_0| \, dx  \le \e(\|u_n\|_2\|u_n - u_0\|_2 + \|u_n\|_{2^{**}}^{2^{**} -1}\|u_n - u_0\|_{2^{**}}) + c_\e\|u_n\|_p^{p-1}\|u_n - u_0\|_p\\
		 \le C \e
	\]
and so we conclude.
\\
	If $N=4$, for every $\e,\alpha>0$ and $\sigma \ge 2$ there exists $c_\e = c_\e(\alpha, \sigma)>0$ such that for every $s\in\R$
	\[
	|F'(s)| \le \e|s| + c_\e\left(e^{\alpha s^2}-1\right)|s|^{\sigma-1}.
	\]
	Then, taking $M>0$ such that $\|u_n\|\le M$ for every $n$ and $\a<32\pi^2/M^2$, arguing again as in the proof of Proposition \ref{pr:Jarek} we obtain
	\[
	\irn \left(e^{\a u_n^2}-1\right)|u_n|^{\s-1}|u_n - u_0| \, dx \le C\|u_n - u_0\|_{p_3}
	\]
	with $p_3>2$ and $C>0$ not depending on $n$, whence for every sufficiently large $n$
	\[
	\irn |F'(u_n)||u_n-u_0| \, dx \le \eps\|u_n\|_2\|u_n - u_0\|_2 + c_\e C\|u_n - u_0\|_{p_3} \le C\e.
	\]
	Finally, if $N\in\{2,3\}$, fix $p>2$. For every $\e>0$ there exists $c_\e>0$ such that for every $s\in[-T,T]$
	\[
	|F'(s)| \le \e|s| + c_\e|s|^{p-1},
	\]
	where $T>0$ is such that $\|u_n\|_\infty\le T$, and we conclude as in the case $N\ge5$.
\end{proof}

\subsection{Proofs of Theorem \ref{th:main} and Proposition \ref{pr:main}}\label{p+}
\

Following \cite{HIT}, we fix  $m'\in (0,m)$ such that $\beta > -2\sqrt{m'}$, where $m$ is defined in \eqref{g2}, $q\in(2, 2^{**})$, and introduce the functions $h\colon\R\to\R$ and $\overline h\colon\R\to\R$ as
\begin{equation*}
h(s) := \left(m's + g(s)\right)_+ \quad\text{ and } \quad \overline h(s) :=
\begin{cases}
s^{q-1}\sup_{0<t\le s}\frac{h(t)}{t^{q-1}} & \, \text{if } s>0\\
0 & \, \text{if } s=0
\end{cases}
\end{equation*}
for $s\ge0$, extending them oddly for $s<0$. Let us define 
$$H(s) := \int_0^sh(t)\,dt \quad \text{ and } \quad \overline H(s) := \int_0^s\bar h(t)\,dt.$$

In a similar way to \cite[Lemma 2.1, Corollary 2.2]{HIT} we can prove as follows.
\begin{lemma}\label{le:h}
The following properties hold.
\begin{enumerate}[label=(\alph{*}),ref=\alph{*}]
	\item \label{a211} There exists $\delta_0>0$ such that $\overline H(s) = \overline h(s) = H(s) = h(s) = 0$ for every $s\in[-\delta_0,\delta_0]$.
	\item \label{b211} The functions $h$ and $\overline h$ satisfy \eqref{g3}. Moreover, if $N\ge5$, then 
	$$\lim_{s\to+\infty}\frac{\overline H(s)}{s^{2^{**}}} = \lim_{s\to+\infty}\frac{H(s)}{s^{2^{**}}} = 0;$$
	if $N=4$, then for every $\a>0$ $$\lim_{s\to+\infty}\frac{\overline{H}(s)}{e^{\a s^2}} = \lim_{s\to+\infty}\frac{H(s)}{e^{\a s^2}} = 0.$$
	\item \label{c211} For every $s\ge0$, we have that $\overline h(s) \ge h(s) \ge g(s) + m's$ and $\overline H(s) \ge H(s) \ge G(s) + m's^2/2$.
	\item \label{d211} The function $s\mapsto\overline{h}(s)/s^{q-1}$ is non-decreasing on $(0,+\infty)$ and $\overline{h}(s)s \ge q\overline{H}(s) \ge 0$ for all $s\in\R$.
\end{enumerate}
\end{lemma}

Note that, in view of Lemma \ref{le:h},  $\overline h$ and $\overline H$ are well defined  and there holds
\begin{subequations}\label{boundh}
\begin{alignat}{2}
	&\exists C>0 \text{ such that } \overline{h}(s) \le C s^{2^{**}-1} \hbox{ for }  s \ge 0,
	&&\text{ if } N\ge5, \label{boundha}\\
	&\forall \a>0, \s \ge 2\, \exists C>0 \text{ such that } \overline{h}(s) \le C \big(e^{\alpha s^{2}}-1\big)s^{\s-1} \hbox{ for }  s \ge 0,\ \ 
	&&\text{ if } N=4, \label{boundhb}\\
	&\forall T>0, \s \ge 2\, \exists C>0 \text{ such that } \overline{h}(s) \le C s^{\s-1}  \hbox{ for } s \in [0,T],
	&&\text{ if } N\in \{2,3\}, \label{boundhc}
\end{alignat}
\end{subequations}
and
\begin{subequations}\label{boundH}
\begin{alignat}{2}
	&\exists C>0 \text{ such that } \overline{H}(s) \le C|s|^{2^{**}} \hbox{ for } s\in \R,
	&&\text{ if } N\ge5, \label{boundHa} \\
	&\forall \a>0, \s \ge 2\, \exists C>0 \text{ such that } \overline{H}(s) \le C \big(e^{\alpha s^{2}}-1\big)|s|^{\s} \hbox{ for } s\in \R,\ \
	&& \text{ if } N=4, \label{boundHb} \\
	&\forall T>0, \s \ge 2\, \exists C>0 \text{ such that } \overline{H}(s) \le C |s|^{\s} \hbox{ for }  s \in [-T,T],\ \ 
	&&\text{ if } N\in \{2,3\}. \label{boundHc}
\end{alignat}
\end{subequations}
The very same estimates hold for $h$ and $H$ respectively.

We then introduce a comparison $C^1$ functional $\overline{I}\colon\Hd\to\R$ as
\[
\overline{I}(u) := \frac12 \irn \left[(\Delta u)^2 + \b |\nabla u|^2 + m'u^2 \right]dx - \irn \overline{H}(u) \, dx .
\]
Now we can prove the following (cf. \cite[Lemmas 2.4 and 2.5]{HIT}).

\begin{proposition}\label{pr:PS2}
The functionals $I$ and $\overline{I}$ satisfy:
\begin{enumerate}[label=(\alph{*}),ref=\alph{*}]
	\item \label{aPS2} $\overline{I} \le I$;
	\item\label{bPS2} there exist $\rho,\mu>0$ such that $I(u) \ge \overline{I}(u) \ge \mu$ for every $\|u\| = \rho$ and $I(u) \ge \overline{I}(u) \ge 0$ for every $\|u\| \le \rho$;
	\item \label{cPS2} for every integer $k\ge1$ there exists an odd map $\gamma_k \in C(\mathbb{S}^{k-1},\cH)$ such that $\overline I\circ\gamma_k \le I\circ\gamma_k < 0$;
	\item \label{dPS2} $\overline{I}$ satisfies the Palais--Smale condition if restricted to $\cH$.
\end{enumerate}
\end{proposition}

\begin{proof}
(\ref{aPS2}) It follows from Lemma \ref{le:h} (\ref{c211}).\\
(\ref{bPS2}) In virtue of point (\ref{aPS2}), it suffices to prove the statement for $\overline{I}$.\\
Assume first $N\ge5$.
From \eqref{boundHa} 
there exists $C>0$ such that, for every $u\in \Hd$,
\[
\overline I(u)\ge \frac 12 \|\Delta u\|_2^2
+\frac \b 2 \|\n u\|_2^2
+ \frac{m'}{2}\|u\|_2^2
-C \|u\|_{2^{**}}^{2^{**}},
\]
so the statement follows from the classical Sobolev embedding.\\
Now let $N=4$. If $\alpha\in(0,32\pi^2)$ and $\sigma>2$, then from Corollary \ref{corTh14RS} and \eqref{boundHb} there exists $C>0$ such that for every $u\in\Hd$ with $\|u\|\le1$
\[\begin{split}
\overline I(u) & \ge \frac 12 \|\Delta u\|_2^2 + \frac \b 2 \|\n u\|_2^2 + \frac{m'}{2}\|u\|_2^2 - C \irn (e^{\a u^2}-1)|u|^\s \, dx\\
& \ge \frac 12 \|\Delta u\|_2^2 + \frac \b2 \|\n u\|_2^2 + \frac{m'}{2}\|u\|_2^2 - C\|u\|_{\frac{\s\tau}{\tau-1}}^\s
\end{split}\]
for some fixed $\tau \in (1,32\pi^2/\a]$. So again the statement follows from the Sobolev embedding.\\
Finally, let $N\in\{2,3\}$ and fix $T>0$ such that $\|u\|_\infty\le T$ for every $u\in\Hd$ with $\|u\|\le1$.
From \eqref{boundHc} with $\sigma=3$, there exists $C>0$ such that, for every $u\in\Hd$ with $\|u\|\le1$,
\[
\overline I(u) \ge \frac 12 \|\Delta u\|_2^2
+\frac \b2 \|\n u\|_2^2
+\frac{m'}{2}\|u\|_2^2
-C \|u\|_3^3
\]
and we conclude as before.\\
(\ref{cPS2})
Again, in view of point (\ref{aPS2}), it is enough to prove the statement for $I$. Arguing in a similar way\footnote{In particular, one can smooth the piecewise affine functions considered therein so that they belong to $H^2(\RN)$.}
to \cite[Proof of Theorem 10]{BL2} (if $\cH := H_{\cO(N)}^2(\RN)$) or \cite[Proof of Lemma 3.4]{jl} (if $\cH := H_X^2(\RN)$), for every integer $k\ge1$ there exists an odd map $\pi_k \in C(\mathbb{S}^{k-1},\cH)$ such that $\irn G\bigl(\pi_k(\xi)\bigr) \, dx \ge 1$ for every $\xi\in\mathbb{S}^{k-1}$. Let $\lambda>0$ and define $\gamma_k(\xi):= \pi_k(\xi)(\cdot/\lambda)$. We have
\[\begin{split}
I\bigl(\gamma_k(\xi)\bigr) & = \frac{\lambda^{N-4}}{2}\|\Delta\pi_k(\xi)\|_2^2 + \frac{\b \lambda^{N-2}}{2}\|\n\pi_k(\xi)\|_2^2 - \lambda^N \irn G\bigl(\pi_k(\xi)\bigr) \, dx\\
& \le \frac{\lambda^{N-4}}{2}\|\Delta\pi_k(\xi)\|_2^2 + \frac{\b\lambda^{N-2}}{2}\|\n\pi_k(\xi)\|_2^2 - \lambda^N,
\end{split}\]
thus the statement holds for sufficiently large $\lambda$.\\
(\ref{dPS2})
Owing to Lemma \ref{le:h} (\ref{d211}), every Palais--Smale sequence for $\overline{I}$ is bounded,
hence the assertion follows from Proposition \ref{pr:Jarek'}.
\end{proof}

Let
\[
\Gamma_k := \{ \gamma\in C(\B^k,\cH) : \gamma \text{ is odd and } \gamma|_{\partial\B^k} = \gamma_k \},
\]
where $\gamma_k\colon\mathbb{S}^{k-1}\to \cH$ is given in Proposition \ref{pr:PS2} (\ref{cPS2}). Observe that $\Gamma_k\ne\emptyset$ because $\bar\gamma_k\in\Gamma_k$, where
\[
\bar{\gamma}_k(\xi) :=
\begin{cases}
	|\xi|\gamma_k\left(\frac{\xi}{|\xi|}\right) & \text{ if } \xi\ne0\\
	0 & \text{ if } \xi=0.
\end{cases}
\]
Define
\[
\sigma_k := \inf_{\gamma \in \Gamma_k}\sup_{\xi\in\B^k}I\bigl(\gamma(\xi)\bigr), \qquad c_k := \inf_{\gamma \in \Gamma_k}\sup_{\xi\in\B^k}\overline{I}\bigl(\gamma(\xi)\bigr).
\]
Using Proposition \ref{pr:PS2} we check that $\sigma_k \ge c_k \ge \mu$ for every $k\ge1$.\\
We will prove that each $\sigma_k$ is a critical value of $I$.

In order to have that $\lim_k \sigma_k = +\infty$ we use the following result.

\begin{proposition}\label{pr:infty}
We have that $\lim_k c_k = +\infty$.
\end{proposition}
\begin{proof}
For every integer $k\ge1$ consider the family of subsets of $\cH$ given by
\[
\Sigma_k := \{ \gamma(\overline{\B^m\setminus Y}) : \gamma\in\Gamma_m, \, m\ge k, \, \R^m\setminus\{0\}\supset Y=\overline{Y}=-Y, \, \mathfrak{g}(Y)\le m-k \},
\]
where $\mathfrak{g}$ is the Krasnosel'skij genus (cf., e.g., \cite[Chapter II, Section 5]{Struwe}). Then we define the sequence of values
\[
d_k := \inf_{A \in \Sigma_k}\sup_{u \in A}\overline{I}(u).
\]
It is clear that $\{d_k\}$ is nondecreasing. Moreover, since $\gamma(\B^k)\in\Sigma_k$ for every $\gamma\in\Gamma_k$ (i.e., taking $m=k$ and $Y=\emptyset$), there holds $c_k\ge d_k$. Finally, in view of Proposition \ref{pr:PS2} (\ref{dPS2}), one can adapt the argument of \cite[Proof of Theorem 9.12]{Rabinowitz}
and obtain that $\lim_kd_k=+\infty$, concluding the proof.
\end{proof}

Following \cite{HIT}, 
we introduce an auxiliary functional $J\in C^1\big(\R\times H^2(\RN),\R\big)$ given by
\begin{equation}\label{J}
J(s,u)=\frac {e^{s(N-4)}} 2 \|\Delta u\|_2^2 + \frac {\b e^{s(N-2)}} 2\|\n u\|_2^2 - e^{sN} \irn G(u) \, dx.
\end{equation}
For all $(s,u)\in \R\times \Hd$,
\begin{align*}
	J(0,u)&=I(u),
	\\
	J (s,u)&=I\big(u (e^{-s} \cdot)\big).
\end{align*}
We equip $\R\times \Hd$ with the standard product norm $\|(s,u)\|_{\R\times \Hd}=(|s|^2+\|u\|^2)^{1/2}$ and define a sequence of minimax values for $J$ as
\[
\widetilde{\sigma}_k := \inf_{\widetilde\gamma \in \widetilde{\Gamma}_k} \max_{\xi\in\B^k} J\bigl(\widetilde{\gamma}(\xi)\bigr),
\]
\[
\widetilde{\Gamma}_k := \left\{ \widetilde{\gamma} = (\widetilde{\gamma}_1,\widetilde{\gamma}_2) \in C(\B^k,\R\times \cH) : \widetilde{\gamma}_1 \text{ is even, } \widetilde{\gamma}_2 \text{ is odd, and } \widetilde{\gamma}|_{\de\B^k} = (0,\gamma_k) \right\},\]
$\gamma_k$ being given in Proposition \ref{pr:PS2} (\ref{cPS2}). Arguing as in \cite[Section 4]{HIT} we have that $\widetilde{\sigma}_k = \sigma_k$ for every $k\ge1$ and the following properties hold.

\begin{proposition}\label{pr:sqJ}
For every integer $k\ge1$ there exists a sequence $\{(s_n,u_n)\} \subset \R \times \cH$
such that
\begin{enumerate}[label=(\arabic{*}),ref=\arabic{*}]
	\item \label{1prop216} $\lim_n s_n =0$;
	\item  \label{2prop216} $\lim_n J(s_n,u_n) = \sigma_k$;
	\item  \label{3prop216} $\lim_n \de_sJ(s_n,u_n) = 0$;
	\item  \label{4prop216} $\lim_n \de_uJ(s_n,u_n) = 0$ in $\cH^*$.
\end{enumerate}
\end{proposition}

We now prove some fundamental properties of the sequence found in Proposition \ref{pr:sqJ}.
\begin{lemma}\label{le:bdd}
If $N\ge3$ or \eqref{AR} is satisfied, then $\{u_n\}$ is bounded, where $\{u_n\} \subset \cH$ is given in Proposition \ref{pr:sqJ}.
\end{lemma}
\begin{proof}
Let us begin with the case $N\ge3$. Since (\ref{2prop216}) and (\ref{3prop216}) of Proposition \ref{pr:sqJ} read explicitly
\begin{align*}
&\frac{e^{(N-4)s_n}}{2}\|\Delta u_n\|_2^2 + \frac{\b e^{(N-2)s_n}}{2}\|\nabla u_n\|_2^2 - e^{Ns_n}\int_{\RN}G(u_n)\,dx \to \sigma_k,
\\
&\frac{N-4}{2}e^{(N-4)s_n}\|\Delta u_n\|_2^2 + \frac{N-2}{2}\b e^{(N-2)s_n}\|\nabla u_n\|_2^2 - Ne^{Ns_n}\int_{\RN}G(u_n)\,dx \to 0,
\end{align*}
we have
\begin{equation}\label{forlater}
2e^{(N-4)s_n}\|\Delta u_n\|_2^2 +\b  e^{(N-2)s_n}\|\nabla u_n\|_2^2 \to N\sigma_k
\end{equation}
and so, taking into account (\ref{1prop216}) of Proposition \ref{pr:sqJ}, 
$\{\|\Delta u_n\|_2\}, \{\|\nabla u_n\|_2\}, \text{ and } \left\{\int_{\RN}G(u_n)\,dx\right\}$ are bounded.
\\
Now we prove $\{u_n\}$ is bounded in $L^2(\RN)$ as well.\\
Assume first that $N\ge4$. By contradiction, let us suppose that, up to a subsequence, $t_n := \|u_n\|_2^{2/N}\to+\infty$ and define $v_n(x) := u_n(t_nx)$. Then
\[
\|v_n\|_2^2 = 1, \quad \|\nabla v_n\|_2^2 = t_n^{2-N} \|\nabla u_n\|_2^2, \quad \|\Delta v_n\|_2^2 = t_n^{4-N}\|\Delta u_n\|_2^2.
\]
Hence $\{v_n\}$ is bounded in $\Hd$. Since $|\nabla v_n| \to 0$ in $L^2(\RN)$, $v_n \rightharpoonup 0$ in $\Hd$.\\
Moreover
\begin{align*}
	&t_n^N\left|e^{(N-4)s_n} t_n^{-4} \|\Delta v_n\|_2^2 
	+ \beta e^{(N-2)s_n} t_n^{-2} \|\n v_n\|_2^2 - e^{Ns_n}  \irn g(v_n)v_n \, dx\right|
	\\
	&\qquad= \left|e^{(N-4)s_n} \|\Delta u_n\|_2^2 + \beta e^{(N-2)s_n} \|\n u_n\|_2^2 - e^{Ns_n} \irn g(u_n)u_n \, dx\right|
	\\
	&\qquad= \left|\de_uJ(s_n,u_n)[u_n] \right|\le \e_n \|u_n\|
	= \e_n \sqrt{t_n^{N-4}\|\Delta v_n\|_2^2 + \beta t_n^{N-2}\|\n v_n\|_2^2 + m' t_n^N}
\end{align*}
where $\e_n := \|\de_uJ(s_n,u_n)\|_* \to 0$ (due to Proposition \ref{pr:sqJ} (\ref{4prop216})) and $\|\cdot\|_*$ is the norm in $\cH^*$ induced by $\|\cdot\|$, obtaining
\[
\delta_n := e^{(N-4)s_n} t_n^{-4} \|\Delta v_n\|_2^2 + \b e^{(N-2)s_n} t_n^{-2} \|\n v_n\|_2^2 - e^{Ns_n} \irn g(v_n)v_n \, dx \to 0.
\]
Hence, in view of Proposition \ref{pr:Jarek'} with $F' = h$, Lemma \ref{le:h}, and (\ref{g1}), for $n$ large we have
\begin{align*}
\frac{m'}{2} & \le e^{(N-4)s_n} t_n^{-4} \|\Delta v_n\|_2^2 + \beta e^{(N-2)s_n} t_n^{-2} \|\n v_n\|_2^2 + m' e^{Ns_n}\\
& = e^{Ns_n} \irn [m' v_n^2 + g(v_n)v_n] \, dx + \delta_n\\
& \le e^{Ns_n} \irn h(v_n)v_n \, dx + \delta_n \to 0,
\end{align*}
which is a contradiction.\\
If $N=3$, since $\cd(\RT) \hookrightarrow L^\infty(\R^3)$ from Corollary \ref{coemb}, there exists $T>0$ such that $\|u_n\|_\infty\le T$ for every $n$. From Lemma \ref{le:h} (\ref{c211}) and \eqref{boundHc}, there exists $C>0$ such that
\[
\frac{m'}{2}\|u_n\|_2^2 + \irn G(u_n) \, dx \le \irn H(u_n) \, dx \le C \|u_n\|_6^6 \le C \|u_n\|_{\cd}^6
\]
and so, in particular, $\|u_n\|_2$ is bounded.\\
Finally, let us consider the case when \eqref{AR} holds. Observe that
\[
g(s)s - \gamma G(s) \ge m \left(\frac{\gamma}{2}-1\right) s^2 \ge m' \left(\frac{\gamma}{2}-1\right) s^2,
\]
hence \eqref{AR} still holds with $m'$ instead of $m$. Thus, for every sufficiently large $n$, there holds
\begin{align*}
\s_k+ 1 + \|u_n\| 
& \ge
J(s_n,u_n) - \frac1\gamma \de_uJ(s_n,u_n)[u_n]\\
&=
 \left(\frac12 - \frac1\gamma\right)
 [e^{s_n(N-4)} \|\Delta u_n\|_2^2 
+\beta e^{s_n(N-2)}\|\n u_n\|_2^2 
+ m' e^{s_nN} \|u_n\|_2^2]\\
&\quad
+ \frac1\gamma\irn e^{s_nN} \left[g(u_n)u_n -\gamma  G(u_n) - m' \left(\frac{\gamma}{2}-1\right) u_n^2 \right] \, dx\\
& \ge
\frac12 \left(\frac12 - \frac1\gamma\right) \|u_n\|^2
\end{align*}
and we conclude.
\end{proof}


\begin{lemma}\label{le:cpt}
If $N\ge3$ or \eqref{AR} is satisfied, then $\{u_n\}$ contains a convergent subsequence, where $\{u_n\} \subset \cH$ is given in Proposition \ref{pr:sqJ}.
\end{lemma}
\begin{proof}
Since $\{u_n\}$ is bounded in $\Hd$ from Lemma \ref{le:bdd}, there exists $u_0\in H^2(\RN)$ such that, up to a subsequence, $u_n\weakto u_0$ in $\Hd$ and $u_n(x) \to u_0(x)$ for a.e. $x \in \RN$. Moreover, from (\ref{1prop216}) and (\ref{4prop216}) of Proposition \ref{pr:sqJ}, we easily see that $u_0$ is a solution to \eqref{bilap}; in particular,
\begin{equation}\label{nehari}
\|u_0\|^2 = \irn \left(m'u^2_0 + g(u_0)u_0\right)  dx.
\end{equation}
Again from (\ref{4prop216}) of Proposition \ref{pr:sqJ} and the boundedness of $\{u_n\}$ we obtain $$e^{(N-4)s_n}\|\Delta u_n\|_2^2 +\b  e^{(N-2)s_n}\|\n u_n\|_2^2 - e^{Ns_n}\irn g(u_n)u_n \, dx \to 0,$$ whence
\begin{equation}\label{3righe}
\begin{split}
&e^{(N-4)s_n}\|\Delta u_n\|_2^2 +\b  e^{(N-2)s_n}\|\nabla u_n\|_2^2 + m' e^{Ns_n}\|u_n\|_2^2 
\\
&\qquad= e^{Ns_n}\irn [m' u_n^2 + g(u_n)u_n] \, dx + o_n(1)
\\
&\qquad= e^{Ns_n}\irn h(u_n)u_n \, dx
- e^{Ns_n}\irn [h(u_n)u_n - m' u_n^2 - g(u_n)u_n] dx + o_n(1).
\end{split}
\end{equation}
From Proposition \ref{pr:Jarek'} with $F'= h$,
\begin{equation}\label{uno}
\irn h(u_n)u_n \, dx \to \irn h(u_0)u_0 \, dx,
\end{equation}
while from Fatou's Lemma and Lemma \ref{le:h} (\ref{c211}),
\begin{equation}\label{due}
\liminf_n \irn \left(h(u_n)u_n - m' u_n^2 - g(u_n)u_n \right) dx \ge \irn \left(h(u_0)u_0 - m' u_0^2 - g(u_0)u_0 \right) dx.
\end{equation}
Therefore, in virtue of \eqref{nehari}, \eqref{3righe}, \eqref{uno}, and \eqref{due}, $\limsup_n\|u_n\| \le \|u_0\|$ and we conclude that $u_n\to u_0$ in $\Hd$.
\end{proof}


Now we are ready to conclude this section.

\begin{proof}[Proofs of Theorem \ref{th:main} and Proposition \ref{pr:main}]
Fix $k\ge1$: we prove that $\s_k = \widetilde{\sigma}_k$ is a critical value of $I$. Let $\{(s_n,u_n)\} \subset \R \times \cH$ be the sequence from Proposition \ref{pr:sqJ}: from Lemma \ref{le:cpt}, there exists $u_0 \in \cH$ such that $u_n \to u_0$ along a subsequence. Recalling that $s_n \to 0$, there holds
\[
I(u_0) = J(0,u_0) = \sigma_k \quad \text{and} \quad I'(u_0) = \partial_uJ(0,u_0) = 0.\qedhere
\]
\end{proof}

\begin{remark}
Under the assumption \eqref{AR}, we do not need the comparison functional $\bar I$,  because we can prove directly that the sequence $\{\s_k\}$ diverges positively. 
\end{remark}

\section{The zero mass case}\label{3}

We recall that, throughout this section, $N\ge3$.

\subsection{The functional framework}
\
	
Let us start recalling the well-known Radial Lemma.
	\begin{lemma}[\!\!\!\!\!\;\;{\cite[Radial Lemma A.III]{BL}}]\label{le:radialBL}
		If $u\in \mathcal{D}^{1,2}(\RQ)$ is  radially symmetric, then 
		\begin{equation*}
			|u(x)|\le \frac{1}{\sqrt{2}\pi}\frac{\|\n u \|_2}{|x|},\qquad
			\text{for a.e. }x\in \RQ.
		\end{equation*}
	\end{lemma}

	We have the following property.
	\begin{lemma}\label{le:radiallemma}
		There exists $C>0 $ such that  for every $u\in \cd_{\cO(4)}(\RQ)$ with $\|u\|_{\cd}\le 1$,
		\begin{equation*}
			\int_{\R^4}\left(e^{32\pi^2 u^2}-1-32\pi^2 u^2\right)dx\le C.
		\end{equation*}
	\end{lemma}
	
	\begin{proof}
		Fix $u\in \cd_{\cO(4)}(\RQ)$, with $\|u\|_{\cd}\le 1$. For $R>0$ we have
		\begin{multline*}
			\int_{\R^4}\left(e^{32\pi^2 u^2}-1-32\pi^2 u^2\right)dx
			\\
			=\underbrace{\int_{B_{R}}\left(e^{32\pi^2 u^2}-1-32\pi^2 u^2\right)dx}_{I_1}
			+\underbrace{\int_{B_R^c}\left(e^{32\pi^2 u^2}-1-32\pi^2 u^2\right)dx}_{I_2}.
		\end{multline*}
		Let us start with $I_1$. We define a radial function $v(x)=v(|x|)$ as
		\[
		v(|x|):=u(|x|)-u(R), \qquad\text{ for }|x|\le R.
		\]
		Observe that $v\in W^{2,2}(B_R)\cap W^{1,2}_0(B_R)$. 
		Following \cite[Page 655]{RS} and using Lemma \ref{le:radialBL}, we have
		\[
		u^2(|x|)\le v^2(|x|)\left(1+\frac{1}{2 \pi^2}\frac{\|\n u \|_2^2}{R^2}\right)
		+\left(1+\frac{1}{2 \pi^2}\frac{\|\n u \|_2^2}{R^2}\right), \qquad\text{ for }0<|x|\le R.
		\]
		Setting
		\[
		d(R):=1+\frac{1}{2 \pi^2}\frac{\|\n u \|_2^2}{R^2} \quad
		\text{and} \quad
		w(|x|):=v(|x|)\sqrt{1+\frac{1}{2 \pi^2}\frac{\|\n u \|_2^2}{R^2}},
		\]
		we have that $ w\in W^{2,2}(B_R)\cap W^{1,2}_0(B_R)$ and
		\[
		u^2(|x|)\le w^2(|x|)+d(R), \qquad\text{ for }0<|x|\le R.
		\]
		Since $\|u\|_{\cd}\le 1$ and so $\|\Delta u\|_2<1$, being
		\[
		\|\Delta w\|_{L^2(B_R)}^2=d(R)\|\Delta v\|_{L^2(B_R)}^2
		=d(R)\|\Delta u\|_{L^2(B_R)}^2,
		\]
		for $R$ sufficiently large, we deduce that $\|\Delta w\|_{L^2(B_R)}\le 1$.
		So we can apply \cite[Theorem 3.1]{RS} (see also \cite{Tarsi}) deducing the existence of $C=C(R)>0$ such that
		\[
		\int_{B_{R}} e^{32\pi^2 w^2}dx\le C.
		\]
		Hence
		\[
		\int_{B_{R}} e^{32\pi^2 u^2}dx\le e^{32\pi^2d(R)}\int_{B_{R}} e^{32\pi^2 w^2}dx\le C,
		\]
		and this concludes the estimate of $I_1$.
		\\
		Now we focus our attention on $I_2$.
		Using the power series expansion we have
		\[
		I_2=\sum_{k=2}^{+\infty}\frac{(32\pi^2)^k}{k!}I_{2,k},
		\qquad \text{where } I_{2,k}:=\int_{B_R^c}|u|^{2k}dx.
		\]
		For any $k\ge 3$, using again Lemma \ref{le:radialBL}, we have
		\begin{equation*}
			I_{2,k}\le \int_{B_R^c}\frac{1}{(2\pi^2)^{k}}\frac{\|\n u \|_2^{2k}}{|x|^{2k}}dx
			= \frac{\|\n u \|_2^{2k}}{(2\pi^2)^{k-1}} \int_{R}^{+\infty}\frac{d\rho}{\rho^{2k-3}}
			=\frac{\|\n u \|_2^{2k}}{(2\pi^2)^{k-1}} \frac{R^{4-2k}}{2k-4}.
		\end{equation*}
		Therefore, since $\|u\|_{\cd}\le 1$, there exists $C>0$ independent of $u$ such that
		\begin{align*}
			I_2&
			\le 2^9 \pi^4\|u\|_4^4
			+\sum_{k=3}^{+\infty}\frac{(32\pi^2)^k}{k!}\frac{\|\n u \|_2^{2k}}{(2\pi^2)^{k-1}} \frac{R^{4-2k}}{2k-4}
			&\le 2^9 \pi^4\|u\|_4^4
			+\pi^2 R^4\sum_{k=3}^{+\infty}\frac{1}{k!}\left[\frac{16 \|\n u \|_2^{2}}{R^{2}}\right]^k\le C.
		\end{align*}
		This proves our claim.
	\end{proof}

\begin{corollary}\label{corTh14RS-zero}
Let $\sigma\ge 4$, $M>0$, and $\alpha>0$ such that $\alpha M^2<32\pi^2$. Then there exists $C>0$ such that for every $\tau\in \left(1,32\pi^2/(\alpha M^2)\right]$ and every $u\in \cd_{\cO(4)}(\RQ) $ with $\|u\|_{\cd}\le M$,
\[
\int_{\RQ} |u|^\sigma \left(e^{\alpha u^2}-1-\a u^2\right)\,dx\le C\|u\|_\frac{\sigma\tau}{\tau-1}^\sigma.
\]
\end{corollary}
\begin{proof}
First observe that, if $s\ge0$ and $t\ge1$,
\begin{equation*}
(e^s-1-s)^t\le e^{st}-1-st.
\end{equation*}
Then, by H\"older inequality
\begin{align*}
	\int_{\RQ} |u|^\sigma\left(e^{\alpha u^2}-1-\alpha u^2\right)  dx
	&\le
	\| u\|_{\frac{\sigma\tau}{\tau-1}}^\sigma
	\left(\int_{\RQ} \left(e^{\alpha u^2}-1-\alpha u^2\right)^\tau  dx\right)^{1/\tau}
\\
&\le
	\| u\|_{\frac{\sigma\tau}{\tau-1}}^\sigma
	\left(\int_{\RQ} \left(e^{\alpha \tau u^2}-1-\alpha \tau u^2\right)  dx\right)^{1/\tau}.
\end{align*}
Now the arguments are similar to those of the proof of Corollary \ref{corTh14RS}, using Lemma \ref{le:radiallemma}.
\end{proof}

\begin{remark}
Corollary \ref{corTh14RS-zero} remains valid for $0<\sigma<4$ provided $\frac{\s\tau}{\tau-1}\ge4$.
\end{remark}

\subsection{Some compactness results}\label{cacca}
\

\begin{lemma}\label{le:Lions0}
Let $N\ge3$, and $F\colon\R\to\R$ be a continuous function such that
\begin{equation}
\label{Fin00}
\lim_{s \to 0} \frac{F(s)}{|s|^{2^*}} = 0
\end{equation}
and
\begin{alignat}{2}
&\lim_{|s|\to+\infty}\frac{F(s)}{|s|^{2^{**}}} = 0&&\text{ if } N\ge5,\label{Finfty50}\\
&\lim_{|s|\to+\infty}\frac{F(s)}{e^{\alpha s^2}} = 0 \ \text{ for all } \alpha>0 && \text{ if }N=4.\label{Finfty40}
\end{alignat}
Assume that $\{u_n\}\subset\cd(\RN)$ is bounded and there exists $r>0$ such that
\begin{equation*}
\lim_n\sup_{y\in\RN}\int_{B(y,r)}u_n^2 \, dx = 0.
\end{equation*}
If $N=4$, assume additionally that $\{u_n\} \subset \cd_{\cO(4)}(\RQ)$.\\
Then
\[
\lim_n\irn |F(u_n)| \, dx = 0.
\]
\end{lemma}
\begin{proof}
First, let us consider the case $N\ge 5$.\\
By \eqref{Fin00} and \eqref{Finfty50}, for every $p\in(2^*,2^{**})$ and $\e>0$ there exists $c_\e>0$ such that, for all $s\in\R$,
\begin{equation}\label{estFgeq50}
|F(s)| \le \e(|s|^{2^*} + |s|^{2^{**}}) + c_\e|s|^p.
\end{equation}
Since $\{u_n\}$ is bounded in $\cd(\RN)$, by Corollary \ref{coemb}, it is also bounded in $L^{2^{*}}(\RN)$ and $L^{2^{**}}(\RN)$. Then, there exists $C>0$ such that for every $n$,
\[
\irn |F(u_n)| \, dx \le C\e + c_\e\|u_n\|_p^p.
\]
Thus, arguing as in Lemma \ref{le:Lions}, it suffices to prove that $u_n\to0$ in $L^p(\RN)$ at least for one $p\in (2^*,2^{**})$. Let us take $p = 2(N+2)/(N-2)$.\\
From the interpolation inequality for Lebesgue spaces  we have that, for every $y\in\RN$,
\[
\|u_n\|_{L^p(B(y,r))}
\le \|u_n\|_{L^2(B(y,r))}^{1-\lambda}\|u_n\|_{L^{2^{**}}(B(y,r))}^\lambda,
\]
with $\lambda=2^*/p=N/(N+2)$.\\
As in Corollary \ref{coemb}, since $2^{**}=(2^*)^*$, by the Sobolev embeddings, we get
\begin{align*}
\|u_n\|_{L^{2^{**}}(B(y,r))}
&\le C\|u_n\|_{\W^{1,2^*}(B(y,r))},
\end{align*}
where $C>0$ does not depend on $y\in\RN$.\\
Hence
\[
\|u_n\|_{L^p(B(y,r))}^p
\le c\|u_n\|_{L^2(B(y,r))}^{p-2^*} \|u_n\|_{\W^{1,2^*}(B(y,r))}^{2^*},
\]
Then, covering $\RN$ with balls of radius $r$ such that each point is contained in at most $N+1$ balls and using Proposition \ref{premb} we obtain
\[
\|u_n\|_p^p
\le C \sup_k\|u_k\|_{\cd}^{2^*} \sup_{y\in\RN}\left(\int_{B(y,r)}u_n^2 \, dx\right)^{(p-2^*)/2}\to0.
\]
If $N=4$, by \eqref{Fin00} and \eqref{Finfty40},  for every $\e>0$, $\alpha>0$, and $\sigma \ge 4$ there exists $c_\e>0$ such that, for all $s\in\R$,
\begin{equation*}
|F(s)| \le \e s^4 + c_\e|s|^\sigma(e^{\alpha s^2} - 1 - \a s^2).
\end{equation*}
Then, applying Corollaries \ref{coemb} and \ref{corTh14RS-zero}, the boundedness of $\{u_n\} \subset \cd_{\cO(4)}(\RQ)$ implies that for $\alpha>0$ and $\tau>1$ such that $\alpha\tau\sup_{n}\|u_n\|^2 \le 32\pi^2$
\[
\int_{\RQ} |F(u_n)| \, dx
\le
C \varepsilon+c_\e \|u_n\|_\frac{\sigma\tau}{\tau-1}^\sigma,
\]
and so it suffices to prove that $u_n\to0$ in $L^\frac{\sigma\tau}{\tau-1}(\RN)$ at least for one couple $(\sigma,\tau)$ with $\sigma \ge 4$ and $\tau>1$. Let us take, for instance, $\tau=5$ and $\sigma=4$.
Arguing as before, by interpolation we have that for every $y\in\RN$,
\[
\|u_n\|_{L^5(B(y,r))} \le \|u_n\|_{L^2(B(y,r))}^{1-\lambda} \|u_n\|_{L^8(B(y,r))}^\lambda \le C \|u_n\|_{L^2(B(y,r))}^{1-\lambda} \|u_n\|_{\W^{1,4}(B(y,r))}^\lambda,
\]
where $C>0$ does not depend on $y$ and $\lambda=4/5$, which allows us to conclude that $\|u_n\|_{\frac{\sigma\tau}{\tau-1}} = \|u_n\|_5 \to 0$.\\
Finally, if $N=3$, by \eqref{Fin00} , if $p>6$, using the boundedness of $\{u_n\}$ and Corollary \ref{coemb}, we can write
\begin{equation*}
|F(s)| \le \e s^2 + c_\e|s|^p \quad \text{for all } s\in\left[-\sup_n\|u_n\|_\infty,\sup_n\|u_n\|_\infty\right],
\end{equation*}
and so
\[
\irn |F(u_n)| \, dx \le C\e + c_\e\|u_n\|_p^p.
\]
Let us take, for instance, $p=8$. Arguing as before, by interpolation we have that for every $y\in\RN$,
\[
\|u_n\|_{L^8(B(y,r))} \le \|u_n\|_{L^2(B(y,r))}^{1-\lambda} \|u_n\|_{L^\infty(B(y,r))}^\lambda \le C \|u_n\|_{L^2(B(y,r))}^{1-\lambda} \|u_n\|_{\W^{1,6}(B(y,r))}^\lambda
\]
where $C>0$ does not depend on $y$ and $\lambda=3/4$, which allows us to conclude that $\|u_n\|_8 \to 0$.
\end{proof}

Now we present the analogues of Propositions \ref{pr:Jarek}, \ref{pr:Jarek'}, and Corollary \ref{co:emb} for $\cd(\RN)$.

\begin{proposition}\label{pr:JarekD}
Let $N\ge3$ and $F\in C^1(\RN)$ be such that $F(0)=0$ and
\begin{itemize}
	\item if $N\ge5$, then there exists $C>0$ such that
	\begin{equation*}
		|F'(s)| \le C\left(|s|^{2^*-1} + |s|^{2^{**}-1}\right) \quad \text{for all } s\in\R;
	\end{equation*}
	\item if $N=4$, then for every $\alpha>0$ there exist $\sigma \ge 4$ and $C>0$ such that
	\begin{equation*}
	|F'(s)| \le C\left(|s|^3 + \big(e^{\alpha s^2} - 1 - \alpha s^2\big)|s|^{\sigma-1}\right) \quad \text{for all } s\in\R;
	\end{equation*}
	\item if $N=3$, then there exists $C>0$ such that
	\begin{equation*}
		|F'(s)| \le C|s|^5 \quad \text{for all } s\in[-1,1].
	\end{equation*}
\end{itemize}
Let $\{u_n\}\subset \cd(\RN)$ bounded such that $u_n \to u_0$ a.e. in $\RN$ for some $u_0\in\cd(\RN)$. If $N=4$, assume additionally that $\{u_n\} \subset \cd_{\cO(4)}(\RQ)$. Then
\begin{equation}\label{conv1D}
	\lim_n \irn \big(F(u_n) - F(u_n-u_0) \big) dx = \irn F(u_0) \, dx.
\end{equation}
If, in addition,
\begin{alignat*}{2}
	& \lim_{s \to 0} \frac{F(s)}{|s|^{2^*}} = \lim_{|s|\to+\infty} \frac{F(s)}{|s|^{2^{**}}} = 0 && \text{ when } N\ge5,\\
	& \lim_{s \to 0} \frac{F(s)}{s^4} = \lim_{|s|\to+\infty} \frac{F(s)}{e^{\alpha s^{2}}} = 0 \ \ \text{ for all } \alpha>0 \ \ \ && \text{ when } N=4,\\
	& \lim_{s \to 0} \frac{F(s)}{s^6} = 0 && \text{ when } N=3,
\end{alignat*}
and $u_0$ and all the $u_n$ are $\cO$-invariant for a suitable\footnote{If $N=4$, then necessarily $\cO = \cO(4)$ due to the first part.} subgroup $\cO\subset \cO(N)$ compatible with $\RN$, then
\begin{equation}\label{conv2D}
	\lim_n \irn F(u_n) \, dx = \irn F(u_0) \, dx.
\end{equation}
\end{proposition}
\begin{proof}
The proof is similar to that of Proposition \ref{pr:Jarek}, hence we only highlight the differences.\\
If $N\ge5$, likewise we prove that there exists $C>0$ such that for every measurable $\Omega\subset\R^N$, every $t\in[0,1]$, and every $n$
\[
\int_\Omega \big|F'\bigl(u_n + (t-1)u_0\bigr)u_0\big| \, dx \le C\left(\|u_0\|_{L^{2^*}(\Omega)} + \|u_0\|_{L^{2^{**}}(\Omega)}\right).
\]
If $N=4$, then taking $M>0$ such that $\||u_n| + |u_0|\|_{\cd} \le M$ for every $n$, $0 < \alpha < 32\pi^2/M^2$, $\s \ge 4$, $p_1,p_2,p_3 > 1$ such that $1/p_1 + 1/p_2 + 1/p_3 = 1$, $\a M^2 p_1 \le 32\pi^2$, $p_2 \ge 4/(\sigma-1)$, and $p_3 \ge 4$, and using Lemma \ref{le:radiallemma} 
instead of Lemma \ref{Th14RS}, we prove similarly that there exists $C>0$ such that for every measurable $\Omega\subset\R^N$, every $t\in[0,1]$, and every $n$
\[
\int_\Omega \big|F'\bigl(u_n + (t-1)u_0\bigr)u_0\big| \, dx \le C\left(\|u_0\|_{L^4(\Omega)} + \|u_0\|_{L^{p_3}(\Omega)}\right).
\]
If $N=3$, then again we prove in a similar way that there exist $T>0$ such that $\sup_n\|u_n\|_\infty \le T$ and $\widetilde{C} = \widetilde{C}(T) > 0$ such that $|F'(s)| \le \widetilde{C}|s|^5$ for all $s\in[-2T,2T]$ and, consequently, that there exists $C>0$ such that for every measurable $\Omega\subset\R^N$, every $t\in[0,1]$, and every $n$
\[
\int_\Omega \big|F'\bigl(u_n + (t-1)u_0\bigr)u_0\big| \, dx \le C\|u_0\|_{L^6(\Omega)}.
\]
Regardless of the dimension, we prove the first statement using Vitali's Theorem and the second one using Lemmas \ref{le:Lions0} and \ref{le:concen}.
\end{proof}

\begin{remark}
In Subsection \ref{cacchina}, we will see that, when $N=4$, if $F(s)$ is controlled by $e^{\a s^{4/3}}$  at infinity,  similar results hold also in $D^2_{\cO}(\RQ)$, with $\cO\subset \cO(4)$ compatible with $\RQ$.
\end{remark}

\begin{corollary}\label{co:emb0}
Let $N\ge3$ and $\cO\subset\cO(N)$ a subgroup compatible with $\RN$. Then $\cd_\cO(\RN) \hookrightarrow\hookrightarrow L^p(\RN)$, for every $p\in(2^*,2^{**})$.
\end{corollary}
\begin{proof}
Letting $F(s) = |s|^p$, the statement follows from \eqref{conv1D} and \eqref{conv2D} provided that $N \ne 4$ or $\cO = \cO(N)$. Now assume $N=4$ and $\cO \ne \cO(4)$. Then there exists $C>0$ such that
\[
\int_\Omega \big|F'\bigl(u_n + (t-1)u_0\bigr)u_0\big| \, dx = (p-1) \int_\Omega |u_n + (t-1) u_0|^{p-1} |u_0| \, dx \le C \|u_0\|_{L^p(\Omega)}
\]
for every measurable $\Omega \subset \RQ$, every $t \in [0,1]$, and every $n$, hence we conclude as before.
\end{proof}

Finally, arguing in a similar way to the proofs of Propositions \ref{pr:Jarek'} and \ref{pr:JarekD} and using Corollary \ref{co:emb0}, we obtain the following.

\begin{proposition}\label{pr:JarekD'}
Let $N\ge3$ and $F\in C^1(\RN)$ be such that $F(0)=0$ and
\begin{alignat*}{2}
	& \lim_{s \to 0} \frac{F'(s)}{|s|^{2^*-1}} = \lim_{|s|\to+\infty} \frac{F'(s)}{|s|^{2^{**}-1}} = 0 && \text{ when } N\ge5,\\
	& \lim_{s \to 0} \frac{F'(s)}{|s|^3} = \lim_{|s|\to+\infty} \frac{F'(s)}{e^{\alpha s^{2}}} = 0 \ \text{ for all } \alpha>0 \ &&  \text{ when } N=4,\\
	& \lim_{s \to 0} \frac{F'(s)}{|s|^5} = 0 && \text{ when } N=3,
\end{alignat*}
and let $\{u_n\}$ be a bounded sequence of $\cO$-invariant functions in  $\cd(\RN)$, for a suitable subgroup  $\cO\subset\cO(N)$ compatible with $\RN$, such that $u_n \to u_0$ a.e. in $\RN$ for some $u_0\in\cd(\RN)$. If $N=4$, assume additionally that $\cO=\cO(4)$.\\
Then
\begin{equation*}
	\lim_n \irn F'(u_n)u_n \, dx = \irn F'(u_0)u_0 \, dx.
\end{equation*}
\end{proposition}


\subsection{Proof of Theorem \ref{th:mm0} under assumption \eqref{g2''}}\label{p0}
\

The procedure is similar to the one in the positive mass case, therefore we skip some details.  We recall that, in this subsection, we assume (\ref{g3}), hence we consider $\cD = \cd_X(\RN)$ only when $N\ge6$ and  that $\beta > m = 0$. 
\\
Fix $q\in(2^*,2^{**})$ and define
\[
h(s) := \left(\frac\ell2 s^{2^*-1} + g(s)\right)_+
\quad\text{ for } s\ge0,
\]
extending it oddly for $s<0$, and $\overline{h},H,\overline{H} \colon \R \to \R$ as before. Then a result similar to Lemma \ref{le:h} holds.

\begin{lemma}\label{le:h0}
The following properties are satisfied.
\begin{enumerate}[label=(\alph{*}),ref=\alph{*}]
	\item There exists $\delta_0>0$ such that $\overline H(s) = \overline h(s) = H(s) = h(s) = 0$ for every $s\in[-\delta_0,\delta_0]$.
	\item The functions  $h$ and $\overline h$ satisfy \eqref{g3}. Moreover, if $N\ge5$, then 
	$$\lim_{s\to+\infty}\frac{\overline H(s)}{s^{2^{**}}} = \lim_{s\to+\infty}\frac{H(s)}{s^{2^{**}}} = 0;$$
	if $N=4$, then for every $\a>0$ $$\lim_{s\to+\infty}\frac{\overline{H}(s)}{e^{\a s^2}} = \lim_{s\to+\infty}\frac{H(s)}{e^{\a s^2}} = 0.$$
	\item For every $s\ge0$, we have that $\overline h(s) \ge h(s) \ge g(s) + \ell s^{2^*-1}/2$ and $\overline H(s) \ge H(s) \ge G(s) + \ell s^{2^*}/(2\!\cdot \!2^*)$.
	\item The function $s\mapsto\overline{h}(s)/s^{q-1}$ is non-decreasing on $(0,+\infty)$ and $\overline{h}(s)s \ge q\overline{H}(s) \ge 0$ for all $s\in\R$.
\end{enumerate}
\end{lemma}

Consequently, conditions similar to \eqref{boundha}--\eqref{boundhc} and \eqref{boundHa}--\eqref{boundHc} are satisfied, i.e.,
\begin{alignat*}{2}
&\exists C>0 \text{ such that } \overline{h}(s) \le C s^{2^{**}-1} \hbox{ for }  s \ge 0,
&&\text{ if } N\ge5, \\
&\forall \a>0, \s \ge 4\, \exists C>0 \text{ such that } \overline{h}(s) \le C \big(e^{\alpha s^{2}} - 1 - \a s^2\big)s^{\s-1} \hbox{ for }  s \ge 0,\ \ 
&&\text{ if } N=4, \\
&\forall T>0, \s \ge 6\, \exists C>0 \text{ such that } \overline{h}(s) \le C s^{\s-1}  \hbox{ for } s \in [0,T],
&&\text{ if } N=3,
\end{alignat*}
and
\begin{alignat*}{2}
&\exists C>0 \text{ such that } \overline{H}(s) \le C|s|^{2^{**}} \hbox{ for } s\in \R,
&&\text{ if } N\ge5, \\
&\forall \a>0, \s \ge 4\, \exists C>0 \text{ such that } \overline{H}(s) \le C \big(e^{\alpha s^{2}} - 1 - \a s^2\big)|s|^{\s} \hbox{ for } s\in \R,\ \
&& \text{ if } N=4, \\
&\forall T>0, \s \ge 6\, \exists C>0 \text{ such that } \overline{H}(s) \le C |s|^{\s} \hbox{ for }  s \in [-T,T],\ \ 
&&\text{ if } N=3.
\end{alignat*}
The very same estimates hold for $h$ and $H$ respectively. If we define $\overline{I} \colon \cd(\RN) \to \R$ as
\[
\overline{I}(u) := \frac12 \irn \left[(\Delta u)^2 + \b |\nabla u|^2 + \frac{\ell}{2^*}|u|^{2^*} \right]dx - \irn \overline{H}(u) \, dx,
\]
then the following result holds.
\begin{proposition}\label{pr:PS2zero}
The functionals $I$ and $\overline{I}$ satisfy:
\begin{enumerate}[label=(\alph{*}),ref=\alph{*}]
	\item \label{aPS20} $\overline{I} \le I$;
	\item\label{bPS20} there exist $\rho,\mu>0$ such that $I(u) \ge \overline{I}(u) \ge \mu$ for every $\|u\| = \rho$ and $I(u) \ge \overline{I}(u) \ge 0$ for every $\|u\| \le \rho$;
	\item \label{cPS20} for every integer $k\ge1$ there exists an odd map $\gamma_k \in C(\mathbb{S}^{k-1},\cD)$ such that $\overline I\circ\gamma_k \le I\circ\gamma_k < 0$;
	\item \label{dPS20} $\overline{I}$ satisfies the Palais--Smale condition if restricted to $\cD$.
\end{enumerate}
\end{proposition}

\begin{proof}
We can argue as in  Proposition \ref{pr:PS2}. In particular, we use Corollary \ref{corTh14RS-zero} instead of Corollary \ref{corTh14RS} in point \eqref{bPS20} when $N=4$ and Proposition \ref{pr:JarekD'} instead of Proposition \ref{pr:Jarek'} in point \eqref{dPS20}, while concerning point \eqref{cPS20}, we simply observe that $\cH \hookrightarrow \cD$ and so we can consider the same $\gamma_k$ given in Proposition \ref{pr:PS2}.
\end{proof}

For every integer $k\ge1$ we define
\begin{equation}\label{Gamma0}
\Gamma_k^0 := \{ \gamma\in C(\B^k,\cD) : \gamma \text{ is odd and } \gamma|_{\partial\B^k} = \gamma_k \} \supset \Gamma_k \ne \emptyset,
\end{equation}
where
$\gamma_k$ is given in point \eqref{cPS20} of Proposition \ref{pr:PS2zero}.\\
We also define
\begin{equation}\label{minimax}
\sigma_k := \inf_{\gamma \in \Gamma_k^0}\max_{\xi\in\B^k}I\bigl(\gamma(\xi)\bigr), \qquad c_k := \inf_{\gamma \in \Gamma_k^0}\max_{\xi\in\B^k}\overline{I}\bigl(\gamma(\xi)\bigr).
\end{equation}
Observe that $\sigma_k \ge c_k \ge \mu$,  where
$\mu>0$ has been introduced  in point \eqref{bPS20} of Proposition \ref{pr:PS2zero}, and the analogous of 
Proposition \ref{pr:infty} still holds. Therefore $\lim_k \sigma_k = +\infty$.

Next, we define $J \in C^1\bigl(\R \times \cd(\RN),\R\bigr)$ as in \eqref{J} equipping $\R \times \cd(\RN)$ with the standard product norm $\|(s,u)\|_{\R\times\cd} = (s^2 + \|u\|_{\cd}^2)^{1/2}$ and
\[
\widetilde{\sigma}_k := \inf_{\widetilde\gamma \in \widetilde{\Gamma}_k^0} \max_{\xi\in\B^k} J\bigl(\widetilde{\gamma}(\xi)\bigr),
\]
\[
\widetilde{\Gamma}_k^0 := \left\{ \widetilde{\gamma} = (\widetilde{\gamma}_1,\widetilde{\gamma}_2) \in C(\B^k,\R\times \cD) : \widetilde{\gamma}_1 \text{ is even, } \widetilde{\gamma}_2 \text{ is odd, and } \widetilde{\gamma}|_{\de\B^k} = (0,\gamma_k) \right\},
\]
thus still $\widetilde{\sigma}_k = \s_k$ and Proposition \ref{pr:sqJ} holds, except now $\{u_n\} \subset \cD$ and $\de_uJ(s_n,u_n) \to 0$ in $\cD^*$.

\begin{proof}[Proof of Theorem \ref{th:mm0} under assumption \eqref{g2''}]
Fix $k\ge1$ and consider the corresponding sequence $\{(s_n,u_n)\} \subset \R \times \cD$ given in
 the zero-mass-regime analogous of Proposition \ref{pr:sqJ}. We check that the sequence $\{u_n\}$ is bounded in $\cd(\RN)$ with the same proof as in Lemma \ref{le:bdd} when $N\ge3$ (in fact, the proof is easier because we do not prove that $\|u_n\|_2$ is bounded) and that, up to a subsequence, there exists $u_0 \in \cD$ such that $u_n \to u_0$ in $\cd(\RN)$ arguing as in the proof of Lemma \ref{le:cpt} with $\ell|u|^{2^*}/2$ instead of 
$m'u^2$. Then
\[
I(u_0) = J(0,u_0) = \s_k \quad \text{and} \quad I'(u_0) = \de_uJ(0,u_0) = 0,
\]
and so we conclude.
\end{proof}

\subsection{Proof of Theorem \ref{th:mm0} under assumption \eqref{g2'}}\label{p00}
\

Also in this subsection, we assume (\ref{g3}) and so we consider $\cD = \cd_X(\RN)$ only when $N\ge6$.

The main difference between (\ref{g2'}) and (\ref{g2''}) is that, when the former holds, we can no longer define the function $\overline h$, which in turn is used to prove that the sequence of critical values $\{\s_k\}$ 
defined in \eqref{minimax} diverges positively. 
We follow the approach of \cite{CGT} (see also \cite{HT,IT}).

Let us consider the functional $P \colon \cd(\RN) \to \R$ defined as
\[
P(u) := \frac{N-4}{2} \|\Delta u\|_2^2 + \beta \frac{N-2}{2} \|\nabla u\|^2 -N\irn G(u) \, dx.
\]
Let $b \in \R$. We say that $I$ restricted to $\cD$  satisfies the Palais--Smale--Poho\v{z}aev condition at the level $b$, $(PSP)_b$ for short,
if and only if every sequence $\{u_n\} \subset \cD$ such that
\[
\lim_n I(u_n) = b, \quad \lim_n P(u_n) = 0, \quad \lim_n \|I'(u_n)\|_{\cD,*} = 0
\]
has a (strongly) convergent subsequence, where we recall that $\|\cdot\|_{\cD,*}$ is the dual norm in $\cD^*$ induced by $\|\cdot\|_{\cd}$. We also define
\[
K_b := \left\{u \in \cD : I(u) = b, \, P(u) = 0, \, \|I'(u)\|_{\cD,*} = 0\right\}.
\]
Let us recall the functional $J \colon \R \times \cd(\RN) \to \R$ given by \eqref{J}. By explicit computations we obtain the following result.

\begin{lemma}\label{le:comput}
For every $s,t \in \R$ and every $u,v \in \cd(\RN)$ there holds:
\begin{itemize}
	\item $\de_s J(s,u) = P\bigl(u(e^{-s}\cdot)\bigr)$;
	\item $\de_u J(s,u)[v] = I'\bigl(u(e^{-s}\cdot)\bigr) \left[v(e^{-s}\cdot)\right]$;
	\item $J\bigl(s+t,u(e^{t}\cdot)\bigr) = J(s,u)$.
\end{itemize}
\end{lemma}

In particular, $P(u) = \de_s J(0,u)$. Then, arguing  in  a similar way to  Lemmas \ref{le:bdd} and \ref{le:cpt} with $s_n = 0$,  without the term $m'u^2$,  and $g$ instead of $h$, see also Proof of Theorem \ref{th:mm0} under assumption \eqref{g2''}, we obtain the following.

\begin{proposition}\label{pr:psp}
For every $b \in \R$, the functional $I$, restricted to $\cD$, satisfies $(PSP)_b$.
\end{proposition}

In fact, taking into account \eqref{forlater}, we observe that sequences $\{u_n\}$ such that $I(u_n) \to b$ and $P(u_n) \to 0$ exist only if $b \ge 0$.

Now let us focus on $J$. We consider $\M := \R \times \cD$ as a Hilbert manifold and define the equivalent norm
\[
\|(t,v)\|_s := \|(t,v)\|_{(s,u)} := \left(t^2 + e^{(N-4)s} \|\Delta v\|_2^2 + e^{(N-2)s} \|\nabla v\|_2^2\right)^{1/2}
\]
for every $(t,v) \in T_{(s,u)}\M$, $(s,u) \in \M$. The corresponding dual norm on $T_{(s,u)}^*\M$ will be denoted by $\|\cdot\|_{s,*}$. Moreover, for every $\tau \in \R$, every $(s,u) \in \M$, and every $(t,v) \in T_{(s,u)}\M$
\begin{equation}\label{norm}
\big\|\bigl(t,v(e^\tau\cdot)\bigr)\big\|_{s+\tau} = \|(t,v)\|_s.
\end{equation}
Given $(s,u),(t,v) \in \M$, we define the distance between them as
\[
\fd\bigl((s,u),(t,v)\bigr) := \inf_{\gamma \in \Gamma} \int_0^1 \|\dot\gamma(\tau)\|_{\gamma(\tau)} \, d\tau,
\]
where
\[
\Gamma := \left\{\gamma \in C^1([0,1],\M) : \gamma(0) = (s,u) \text{ and } \gamma(1) = (t,v)\right\}.
\]
In virtue of \eqref{norm}, we obtain that, for every $\tau \in \R$ and every $(s,u),(t,v) \in \M$,
\begin{equation}\label{dist-inv}
\fd\bigl((s,u),(t,v)\bigr) = \fd\Bigl(\bigl(s+\tau,u(e^\tau\cdot)\bigr), \bigl(t+\tau,v(e^\tau\cdot)\bigr)\Bigr).
\end{equation}
As usual, if $(s,u) \in \M$ and $A\subset \M$, then
\[
\fd\bigl((s,u),A\bigr) = \inf_{(t,v) \in A} \fd\bigl((s,u),(t,v)\bigr).
\]
From Lemma \ref{le:comput}, for every $(s,u) \in \M$
\[
\|J'(s,u)\|_{s,*}^2 = \bigl(P\bigl(u(e^{-s}\cdot)\bigr)\bigr)^2 + \big\|I'\bigl(u(e^{-s}\cdot)\bigr)\big\|_{\cD,*}^2.
\]
For $b \in \R$, let
\[
\widetilde{K}_b := \left\{(s,u) \in \M : J(s,u) = b \text{ and } \|J'(s,u)\|_{s,*} = 0\right\} = \left\{\bigl(s,u(e^{-s}\cdot)\bigr) : u \in K_b \text{ and } s \in \R\right\}.
\]
We say that $J$ satisfies $\widetilde{(PS)}_b$ if and only if every sequence $\{(s_n,u_n)\} \subset \M$ such that
\[
\lim_n J(s_n,u_n) = b \quad \text{and} \quad \lim_n \|J'(s_n,u_n)\|_{s_n,*} = 0
\]
has a subsequence (still denoted $(s_n,u_n)$) such that
\[
\lim_n \fd\bigl((s_n,u_n),\widetilde{K}_b\bigr) = 0.
\]
Using Proposition \ref{pr:psp} and arguing as in \cite[Proposition 4.6]{CGT}, we have the following.

\begin{proposition}\label{pr:ps}
The functional $J$ satisfies the $\widetilde{(PS)}_b$ condition at every level $b \in \R$.
\end{proposition}

For $a,b \in \R$ with $a<b$, $r > 0$, $A \subset \cD$, and $B \subset \M$, we denote
\begin{align*}
& I^b  := \left\{u \in \cD : I(u) \le b\right\},
\quad
&&I_a^b  := \left\{u \in I^b : I(u) \ge a\right\},\\
&J^b  := \left\{(s,u) \in \M : J(s,u) \le b\right\},
\quad
&&J_a^b := \left\{(s,u) \in J^b : J(s,u) \ge a\right\},\\
&  N_r(A) := \left\{u \in \cD : \mathrm{dist}(u,A) <r \right\},
\quad
&&\widetilde{N}_r(B)  := \left\{(s,u) \in \M : \fd\bigl((s,u),B\bigr) <r  \right\},
\end{align*}
where
\[
\mathrm{dist}(u,A) := \inf_{v \in A} \|u - v\|_{\cd}.
\]
Note that $N_r(\emptyset) = \emptyset$ and $\widetilde{N}_r(\emptyset) = \emptyset$. The next property is a consequence of Proposition \ref{pr:ps}.

\begin{corollary}\label{cor:ps}
For every $\rho > 0$ there exists $\delta = \delta(\rho) > 0$ such that
\[
(s,u) \in J_{b-\delta}^{b+\delta} \setminus \widetilde{N}_\rho(\widetilde{K}_b) \; \Longrightarrow \; \|J'(s,u)\|_{s,*} \ge \delta.
\]
\end{corollary}

Using Corollary \ref{cor:ps} and arguing as in \cite[Proof of Theorem 7.2]{CGT}, we obtain the following.

\begin{lemma}\label{le:defJ}
Let $b,\bar\eps,r > 0$ and $\widetilde{\cU}:=\widetilde{N}_r(\widetilde{K}_b) $. Then there exist $\eps \in (0,\bar\eps)$ and $\widetilde{\eta} \colon [0,1] \times \M \to \M$ continuous such that:
\begin{itemize}
	\item $\widetilde{\eta}(t,s,u) = (s,u)$, if $t = 0$ or $J(s,u) \le b-\bar\eps$;
	\item $J\bigl(\widetilde{\eta}(t_1,s,u)\bigr) \ge J\bigl(\widetilde{\eta}(t_2,s,u)\bigr)$ if $t_1 < t_2$;
	\item $\widetilde{\eta}(1,J^{b+\eps} \setminus \widetilde{\cU}) \subset J^{b-\eps}$ and $\widetilde{\eta}(J^{b+\eps}) \subset J^{b-\eps} \cup \widetilde{\cU}$;
	\item $\widetilde{\eta}$ is even in $s$ and odd in $u$.
\end{itemize}
\end{lemma}

Define $\fp \colon \M \to \cD$ and $\fii \colon \cD \to \M$ by
\[
\fp(s,u) := u(e^{-s}\cdot) \quad \text{and} \quad \fii(u) = (0,u).
\]
An immediate consequence is that
\[
\fp \circ \fii = \mathrm{id}_{\cD}, \quad J = I \circ \fp, \quad I = J \circ \fii, \quad \fp(\widetilde{K}_b) = K_b.
\]
Moreover, the following holds.

\begin{lemma}\label{le:tech}
Let $b>0$. For every $\rho > 0$ there exists $R = R(\rho) > 0$ such that
\begin{subequations}
\begin{alignat}{1}
\fp\bigl(\widetilde{N}_\rho(\widetilde{K}_b)\bigr) \subset N_R(K_b), \label{piia}\\
\fii\big( N_R(K_b)^c\big) \subset \widetilde{N}_\rho(\widetilde{K}_b)^c. \label{piib}
\end{alignat}
\end{subequations}
Moreover,
\[
\lim_{\rho \to 0^+} R(\rho) = 0.
\]
\end{lemma}

\begin{proof}
Let $(s,u) \in \widetilde{N}_\rho(\widetilde{K}_b)$ and note that \eqref{piia} reads explicitly
\begin{equation}\label{expiia}
\fd\bigl((s,u),\widetilde{K}_b\bigr) < \rho \; \Longrightarrow \; \mathrm{dist}\bigl(\fp(s,u),K_b\bigr) < R.
\end{equation}
Note also that, if $s = 0$, then \eqref{expiia} becomes
\[
\fd\bigl((0,u),\widetilde{K}_b\bigr) < \rho \; \Longrightarrow \; \mathrm{dist}(u,K_b) < R,
\]
which is equivalent to \eqref{piib} and so it is enough to prove \eqref{expiia}.\\
Observe that, 
by \eqref{dist-inv}, $$\fd\bigl((s,u),\widetilde{K}_b\bigr) = \fd\bigl((0,\fp(s,u)),\widetilde{K}_b\bigr).$$
To simplify notations, we relabel $\fp(s,u)$  as $u$. So there exists $\gamma \in C^1([0,1],\M)$ such that $\gamma(0) = (0,u)$, $\gamma(1) \in \widetilde{K}_b$, and
\[
\int_0^1 \|\dot{\gamma}(t)\|_{\gamma(t)} \, dt < \rho.
\]
With a small abuse of notation, let us write $\gamma(t) = \bigl(s(t),u(t)\bigr)$, which yields $s(0) = 0$ and $u(0) = u$. Since $\bigl(s(1),u(1)\bigr) \in \widetilde{K}_b$, we have $\fp\bigl(s(1),u(1)\bigr) \in K_b$, therefore
\begin{equation}\label{I1I2}
\mathrm{dist}(u,K_b) \le \|u - \fp\bigl(s(1),u(1)\bigr)\|_{\cd} \le \underbrace{\|u - u(1)\|_{\cd}}_{I_1} + \underbrace{\|u(1) - \fp\bigl(s(1),u(1)\bigr)\|_{\cd}}_{I_2}.
\end{equation}
Note preliminarily that, since $s(0)=0$, for all $t\in [0,1]$ we have
\[
|s(t)| \le \int_0^1 |\dot{s}(t)| \, dt \le \int_0^1 \|\dot{\gamma}(t)\|_{\gamma(t)} \, dt < \rho.
\]
Therefore
\[
(N-2)s(t) + N \rho > 0, \quad t\in[0,1].
\]
We claim that
\[
(N-4)s(t) + N \rho > 0, \quad t\in[0,1].
\]
Indeed, if $N \ge 5$, we argue as above; if $N = 4$, the claim is obvious; if $N=3$, we simply observe that $s(t) < \rho < 3\rho$.
Then
\begin{align*}
I_1 & \le \int_0^1 \|\dot{u}(t)\|_{\cd} \, dt = \int_0^1 \left(\|\Delta\dot{u}(t)\|_2^2 + \|\nabla\dot{u}(t)\|_2^2\right)^{1/2} \, dt\\
& \le e^{N\rho/2} \int_0^1 \left(\bigl(\dot{s}(t)\bigr)^2 + e^{(N-4)s(t)} \|\Delta\dot{u}(t)\|_2^2 + e^{(N-2)s(t)} \|\nabla\dot{u}(t)\|_2^2\right)^{1/2} \, dt\\
& = e^{N\rho/2} \int_0^1 \|\dot{\gamma}(t)\|_{\gamma(t)} \, dt
< e^{N\rho/2} \rho.
\end{align*}
At the same time,
\[
I_2
=\|\fp\big(-s(1),\fp\big(s(1),u(1)\big)\big) - \fp\big(s(1),u(1)\big) \|_{\cd}
\le \sup\left\{\|\fp(t,v) - v\|_{\cd} : |t| < \rho \text{ and } v \in K_b\right\}.
\]
Combining \eqref{I1I2} with the inequalities for $I_1$ and $I_2$ we infer
\[
\mathrm{dist}(u,K_b)
<
\rho e^{N\rho/2} + \sup\left\{\|\fp(t,v) - v\|_{\cd} : |t| < \rho \text{ and } v \in K_b\right\} =: R(\rho)
\]
and $R(\rho) \to 0$ as $\rho \to 0^+$ thanks to the compactness of $K_b$.
\end{proof}

Using Lemmas \ref{le:defJ} and \ref{le:tech} and arguing as in \cite[Proof of Theorem 7.1]{CGT}, we obtain the following.

\begin{lemma}\label{le:defI}
Let $b,\bar\eps > 0$ and $\cU \subset \cD$ a neighbourhood of $K_b$. Then there exists $\eps \in (0,\bar\eps)$ and $\eta \colon [0,1] \times \cD \to \cD$ such that:
\begin{itemize}
	\item $\eta(t,u) = u$ if $t = 0$ or $I(u) \le b - \bar\eps$,
	\item $I\bigl(\eta(t_1,u)\bigr) \ge I\bigl(\eta(t_2,u)\bigr)$ if $t_1 < t_2$,
	\item $\eta(1,I^{b+\eps} \setminus \cU) \subset I^{b-\eps}$ and $\eta(1,I^{b+\eps}) \subset I^{b-\eps} \cup \cU$,
	\item $\eta$ is odd in $u$.
\end{itemize}
\end{lemma}

\begin{proof}[Proof of Theorem \ref{th:mm0} under assumption \eqref{g2'}]
For every integer $k \ge 1$ consider the family of subsets of $\cD$ given by
\[
\Sigma_k^0 := \left\{\gamma(\overline{\B^m \setminus Y}) : \gamma \in \Gamma_m^0, \, m \ge k, \, \R^m \setminus \{0\} \supset Y = \overline{Y} = -Y, \, \mathfrak{g}(Y) \le m-k\right\},
\]
where $\Gamma_m^0$ is defined in \eqref{Gamma0}, and $\mathfrak{g}$ is the Krasnosel'skij genus. Moreover we define the values
\[
d_k := \inf_{A \in \Sigma_k^0} \sup_{u \in A} I(u).
\]
Then, using Lemma \ref{le:defI}, the compactness of $K_b$, and arguing in a similar way to in \cite[Proof of Theorem 9.12]{Rabinowitz}, we obtain that each $d_k$ is a critical value of $I$ and $d_k \to +\infty$ as $k \to +\infty$.
\end{proof}

\subsection{Proof of Theorem \ref{th:mainnr}}\label{cacchina}
\

We conclude this section looking for non-radial solutions in $\cd_X(\RQ)$. To this aim we have to consider stronger assumptions about the nonlinearity $g$ in order that the functional $I \colon \cd(\RQ) \to \R$ is of class $C^1$ (or even well defined). In particular, we assume \eqref{g3'} instead of \eqref{g3}.
\begin{lemma}[\;\;\!\!\!\!\!{\cite[Lemma 1]{do}}]\label{le:do}
For any $u\in \W^{1,4}(\RQ)$ and $\a >0$, we have
\[
 \int_{\RQ} \left(e^{\a |u|^{4/3} }-1-\a |u|^{4/3}-\frac{\a^2} 2|u|^{8/3}\right)\,dx<+\infty.
\]
Moreover, if $\|\n u\|_4\le 1$, $\|u\|_4\le A < +\infty$, and $\a <4\o_3^{1/3}$, then there exists  $C=C(\a,A)>0$ such that
\[
\int_{\RQ} \left(e^{\a |u|^{4/3} }-1-\a |u|^{4/3}-\frac{\a^2} 2|u|^{8/3}\right)\,dx\le C(\a,A).
\]
\end{lemma}
\begin{corollary}\label{corTh14RS-nr}
Let $\s\ge4$, $M>0$, and $\a>0$ such that $\a M^{4/3} < 4\o_3^{1/3}$. There exists $C>0$ such that for every $\tau \in \left(1,4\o_3^{1/3}/(\a M^{4/3})\right)$ and every $u \in \W^{1,4}(\RQ)$ with $\|u\|_{\W^{1,4}} \le M$
\[
\int_{\RQ}|u|^\sigma \left(e^{\a|u|^{4/3}} - 1 - \alpha|u|^{4/3} - \frac{\a^2}{2}|u|^{8/3}\right) \, dx \le C\|u\|_\frac{\s\tau}{\tau-1}^\s.
\]
\end{corollary}
\begin{proof}
First observe that, if $s\ge0$ and $t\ge1$,
\begin{equation*}
\left(e^s - 1 - s - \frac{s^2}{2}\right)^t\le e^{st} - 1 - st - \frac{(st)^2}{2}.
\end{equation*}
Then, by H\"older inequality
\begin{align*}
&\int_{\RQ} |u|^\sigma \left(e^{\alpha |u|^{4/3}} - 1 - \alpha |u|^{4/3} - \frac{\a}{2}|u|^{8/3}\right)  dx\\
&\qquad	\le \| u\|_{\frac{\s\tau}{\tau-1}}^\sigma
\left(\int_{\RQ} \left(e^{\alpha |u|^{4/3}} - 1 - \alpha |u|^{4/3} - \frac{\a^2}{2}|u|^{8/3}\right)^\tau  dx \right)^{1/\tau}\\
&\qquad	\le \| u\|_{\frac{\s\tau}{\tau-1}}^\sigma
\left(\int_{\RQ} \left(e^{\alpha \tau |u|^{4/3}} - 1 - \alpha \tau |u|^{4/3} - \frac{(\a \tau)^2}{2}|u|^{8/3}\right)  dx\right)^{1/\tau}.
\end{align*}
Now the arguments are similar to those of the proof of Corollary \ref{corTh14RS}, using Lemma \ref{le:do} with $A=1$.
\end{proof}

\begin{proof}[Proof of Theorem \ref{th:mainnr}]
Using Lemma \ref{le:do} and Corollary \ref{corTh14RS-nr}, the analogous of the compactness results in Subsection \ref{cacca} holds for $\cd_X(\RQ)$ under the stronger assumption at infinity coming from \eqref{g3'}.
Therefore, the arguments are similar to those in Subsections \ref{p+} and \ref{p0}  or Subsection \ref{p00}  adapted to $\cd_X(\RQ)$.
\end{proof}

\section{Open problems and related remarks}\label{4}

\textbf{Question:} can we extend Lemma \ref{le:radiallemma} to \textit{all} the functions in $\cd(\RQ)$? That is, does $C>0$ exist such that
\begin{equation*}
\int_{\R^4}\left(e^{32\pi^2 u^2}-1-32\pi^2 u^2\right)dx\le C
\end{equation*}
for every $u\in \cd(\RQ)$ with $\|u\|_{\cd}\le 1$?

A version of Lemma \ref{le:radiallemma} valid in all of $\cd(\RQ)$ but where the constant $32\pi^2$ is replaced with $1$, which is still sufficient for $I$ to be of class $C^1$ in $\cd(\RQ)$, is satisfied provided the following condition holds:
\begin{enumerate}[label=(q\arabic{*}),ref=q\arabic{*}]
	\item \label{q1} 
	$\displaystyle \sum_{k=2}^{+\infty} \frac{C_{2k}^{2k}}{k!} < +\infty$,
\end{enumerate}
where, for $p\ge4$, $C_p>0$ is the best constant in the inequality of Gagliardo--Nirenberg-type
\[
\|u\|_p \le C_p \|\Delta u\|_2^{1-4/p} \|u\|_4^{4/p} \quad \text{for every } u \in \cd(\RQ).
\]
Note that we \textit{do not} need the formula above to hold for every $u \in \{v \in L^4(\RQ) : \Delta v \in L^2(\RQ)\} \supset \cd(\RQ)$ as it would be for the classical Gagliardo--Nirenberg inequality.

\begin{theorem}\label{th:q1}
If \eqref{q1} holds, then there exists $C>0$ such that for every $M>0$ and every $u \in \cd(\RQ)$ with $\max\{\|u\|_4,\|\Delta u\|_2\} \le M$
\[
\int_{\RQ} \left(e^{u^2/M^2} - 1 - \frac{u^2}{M^2}\right) \, dx \le C.
\]
\end{theorem}
Notice that taking $M=1$ we obtain the version of Lemma \ref{le:radiallemma} mentioned before.
\begin{proof}
We can assume $u\ne0$. Observe preliminarily that
\[\begin{split}
\int_{\RQ} \left(e^{u^2/\|\Delta u\|_2^2} - 1 - \frac{u^2}{\|\Delta u\|_2^2}\right) \, dx & = \int_{\RQ} \sum_{k=2}^{+\infty} \frac1{k!} \frac{u^{2k}}{\|\Delta u\|_2^{2k}} \, dx
= \sum_{k=2}^{+\infty} \frac1{k!} \frac{\|u\|_{2k}^{2k}}{\|\Delta u\|_{2}^{2k}}
\\
&  \le \sum_{k=2}^{+\infty} \frac{C_{2k}^{2k}}{k!} \frac{\|u\|_4^4}{\|\Delta u\|_2^4} = C \frac{\|u\|_4^4}{\|\Delta u\|_2^4},
\end{split}\]
where $C := \sum_{k=2}^{+\infty} C_{2k}^{2k}/k! < \infty$. Next, since $\|\Delta u\|_2 \le M$,
\[\begin{split}
e^{u^2/M^2} - 1 - \frac{u^2}{M^2}
& = \sum_{k=2}^{+\infty} \frac{1}{k!} \frac{u^{2k}}{M^{2k}}
= \frac{1}{M^4} \sum_{k=2}^{+\infty} \frac{1}{k!}\frac{u^{2k}}{M^{2k-4}}
\le \frac{1}{M^4} \sum_{k=2}^{+\infty} \frac{1}{k!} \frac{u^{2k}}{\|\Delta u\|_2^{2k-4}}\\
& = \frac{\|\Delta u\|_2^4}{M^4} \sum_{k=2}^{+\infty} \frac{1}{k!} \frac{u^{2k}}{\|\Delta u\|_2^{2k}}
= \frac{\|\Delta u\|_2^4}{M^4} \left(e^\frac{u^2}{\|\Delta u\|_2^2} - 1 - \frac{u^2}{\|\Delta u\|_2^2}\right),
\end{split}\]
thus, since $\|u\|_4 \le M$,
\[\begin{split}
\int_{\RQ} \left(e^{u^2/M^2} - 1 - \frac{u^2}{M^2}\right) \, dx \le \frac{\|\Delta u\|_2^4}{M^4} \int_{\RQ} \left(e^{u^2/\|\Delta u\|_2^2} - 1 - \frac{u^2}{\|\Delta u\|_2^2}\right) \, dx \le C \frac{\|u\|_4^4}{M^4} \le C. \qedhere
\end{split}\]
\end{proof}

Arguing as in Corollary \ref{corTh14RS-zero}, but using Theorem \ref{th:q1} instead of Lemma  \ref{le:radiallemma}, we obtain the following result, which suffices to use the machinery of the previous sections.

\begin{corollary}
If \eqref{q1} holds, then for every $\s\ge4$, $\a>0$, and $M>0$ such that $\a M^2 < 1$ there exists $C>0$ such that for every $\tau \in \bigl(1,1/(\a M^2)\bigr]$ and every $u \in \cd(\RQ)$ with $\max\{\|u\|_4,\|\Delta u\|_2\} \le M$
\[
\int_{\RQ} |u|^\sigma \left(e^{\alpha u^2}-1-\a u^2\right)\,dx\le C\|u\|_\frac{\sigma\tau}{\tau-1}^\sigma.
\]
\end{corollary}

\section*{Acknowledgements}

\noindent The authors are members of GNAMPA (INdAM).\\
Pietro d'Avenia and Alessio Pomponio are partially supported by PRIN 2017JPCAPN {\em Qualitative and quantitative aspects of nonlinear PDEs} and by GNAMPA {\em Modelli EDP nello studio problemi della fisica moderna}.
Jacopo Schino is partially supported by the National Science Centre, Poland (Grant No. 2020/37/N/ST1/00795).

\end{document}